\documentclass[a4paper,11pt, english]{amsart}

\usepackage{amssymb} 
\usepackage{amsfonts}
\usepackage{amsmath}

\usepackage{amsthm}
\usepackage[utf8]{inputenc}
\usepackage{pgf}
\usepackage{graphicx}
\usepackage{tikz}
\usepackage{tikzsymbols}
\usepackage{pdfpages}
\usepackage[all]{xy}
\usepackage{verbatim,enumerate}
\usepackage{hyperref}
\usetikzlibrary{patterns}
\usepackage[yyyymmdd,hhmmss]{datetime}
\usepackage{mathtools}

\renewcommand{\qedsymbol}{\BasicTree[1.3]{orange!97!black}{green!90!white}{green!50!white}{leaf}}

\newtheorem{thm}{Theorem}[section]
\newtheorem{lemma}[thm]{Lemma}
\newtheorem{prop}[thm]{Proposition}
\newtheorem{cor}[thm]{Corollary}
\newtheorem{defi}[thm]{Definition}
\newtheorem*{pb}{Problem}
{\theoremstyle{definition}
\newtheorem{exa}[thm]{Example}
\newtheorem{rem}[thm]{Remark}}

\newcommand{\R}{\mathbb{R}}
\newcommand{\RR}{\mathbb{R}}

\newcommand{\Z}{\mathbb{Z}}
\newcommand{\CP}{\mathbb{C}P}

\newcommand{\N}{\mathbb{N}}
\newcommand{\TP}{\mathbb{T}P}

\newcommand{\C}{\mathbb{C}}

\renewcommand{\epsilon}{\varepsilon}

\renewcommand{\div}{\mathop{\mathrm{div}}}

\newcommand{\Ed}{\text{Edge}}
\newcommand{\val}{\text{val}}
\newcommand{\td}{{\text{''}}}
\newcommand{\tg}{{\text{``}}}

\newcommand{\F}{\mathcal{F}}

\newcommand{\xdim}{m}
\newcommand{\xcod}{k}
\newcommand{\damb}{n}

\title{Planar tropical cubic curves of any genus, and higher dimensional generalisations}
\author{Benoît Bertrand} 
\address{Benoît Bertrand, Institut de mathématiques de Toulouse -- IUT de Tarbes, France}
\email{benoit.bertrand@math.univ-toulouse.fr}

\author{Erwan Brugallé}
\address{Erwan Brugallé, CMLS, École polytechnique, CNRS, Université
  Paris-Saclay, 91128 Palaiseau Cedex, France; Université de Nantes, Laboratoire de
  Mathématiques Jean Leray, 2 rue de la Houssinière, F-44322 Nantes Cedex 3,
France}
\email{erwan.brugalle@math.cnrs.fr}

\author{Lucía López de Medrano}
\address{Lucía López de Medrano, Unidad Cuernavaca del Instituto de Matem\'aticas, UNAM. Mexico.}
\email{lucia.ldm@im.unam.mx}

\date{\today \mbox{ at } \currenttime}

\dedicatory{À la mémoire de notre ami Jean-Jacques Risler, à qui nous n'avons pas eu le temps de raconter ces incongruités.}

\subjclass[2010]{Primary 14T05, 14F45; Secondary 52B20, 52B05}
\keywords{Topology of tropical varieties, tropical Hodge numbers, tropical homology, floor composition}

\begin{document}

\begin{abstract}
  We study the maximal
  values of Betti numbers of tropical subvarieties  of a
  given dimension and degree in $\TP^n$. We provide a lower estimate for the maximal value of the top Betti number, which
  naturally depends on the dimension and degree, but also  on the 
  codimension. In particular, when the codimension is large enough,
  this lower estimate is larger than the 
   maximal value of the corresponding Hodge number of complex algebraic
  projective varieties of the given dimension and degree.
In the case of surfaces, we extend our study to   all tropical
  homology groups.  As a special case, we prove that there exist 
  planar tropical cubic curves of genus $g$ for any non-negative integer $g$.
 \end{abstract} 
\maketitle
\tableofcontents

Throughout the text, we fix a field $\mathbb K$. 
 The $j^{\mbox{th}}$ Betti number $b_j(X)$
 of a topological space $X$ is the dimension of the $j^{\mbox{th}}$ homology group $H_j(X;\mathbb K)$ of $X$ with coefficients in $\mathbb K$.
Otherwise stated, we refer to \cite{BIMS15} for precise
definitions of notions from tropical geometry needed in this text.

\section{Introduction}

\subsection{Curves}
A tropical curve $C$ in $\R^n$ is a piecewise linear graph with
finitely many vertices such that
(see for example \cite{BIMS15,MikRau19}):
\begin{itemize}
\item each edge $e$ of $C$ is equipped with an integer weight
$w_e\in \Z_{>0}$,
and has a directing vector in $\Z^n$;
\item at each vertex $v$ of $C$, adjacent to the edges
$e_1,\cdots,e_l$, the following \emph{balancing condition} is
satisfied:
\[
\sum_{i=1}^l w_{e_i}u_{e_i}=0,
\]
where $u_{e_i}$ is the primitive integer directing vector of $e_i$
pointing away from $v$.
\end{itemize}
Some examples of tropical curves in $\R^2$ are depicted in
Figure \ref{fig:balancing}. A tropical curve is said to be
of \emph{degree} $d$ if 
\[
d=\sum\limits_{e}w_{e}\ {\max}_{j=1}^n\{0,u_{e,j}\},
\]
where the sum ranges over all unbounded edges $e$ of $C$, and
$u_e=(u_{e,1},\cdots, u_{e,n})$ is a primitive integer directing vector 
of $e$ pointing toward infinity, see  \cite{BIMS15}.
\begin{figure}[h]
\centering
\begin{tabular}{ccccc}
\includegraphics[width=3cm, angle=0]{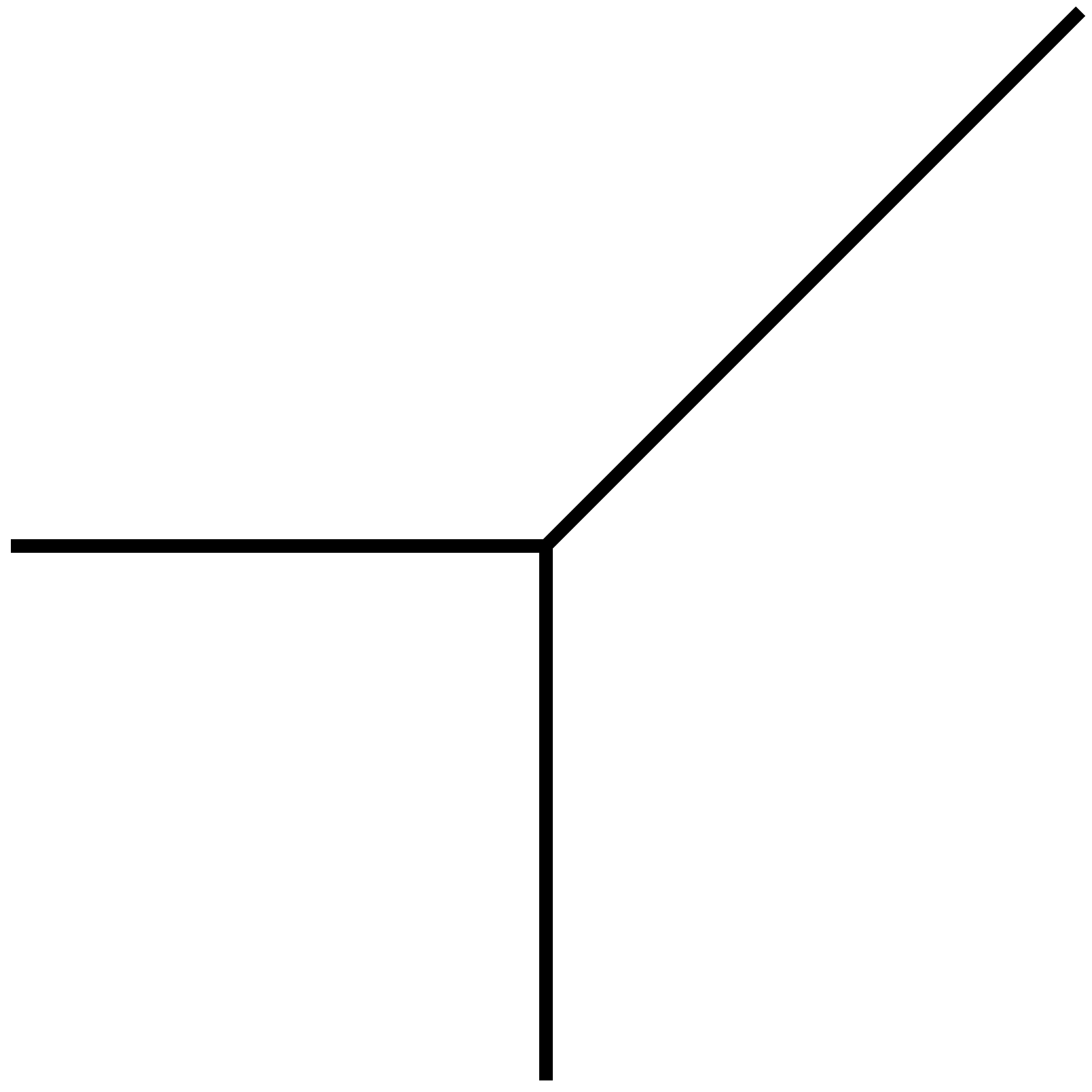} 
&\includegraphics[width=3cm, angle=0]{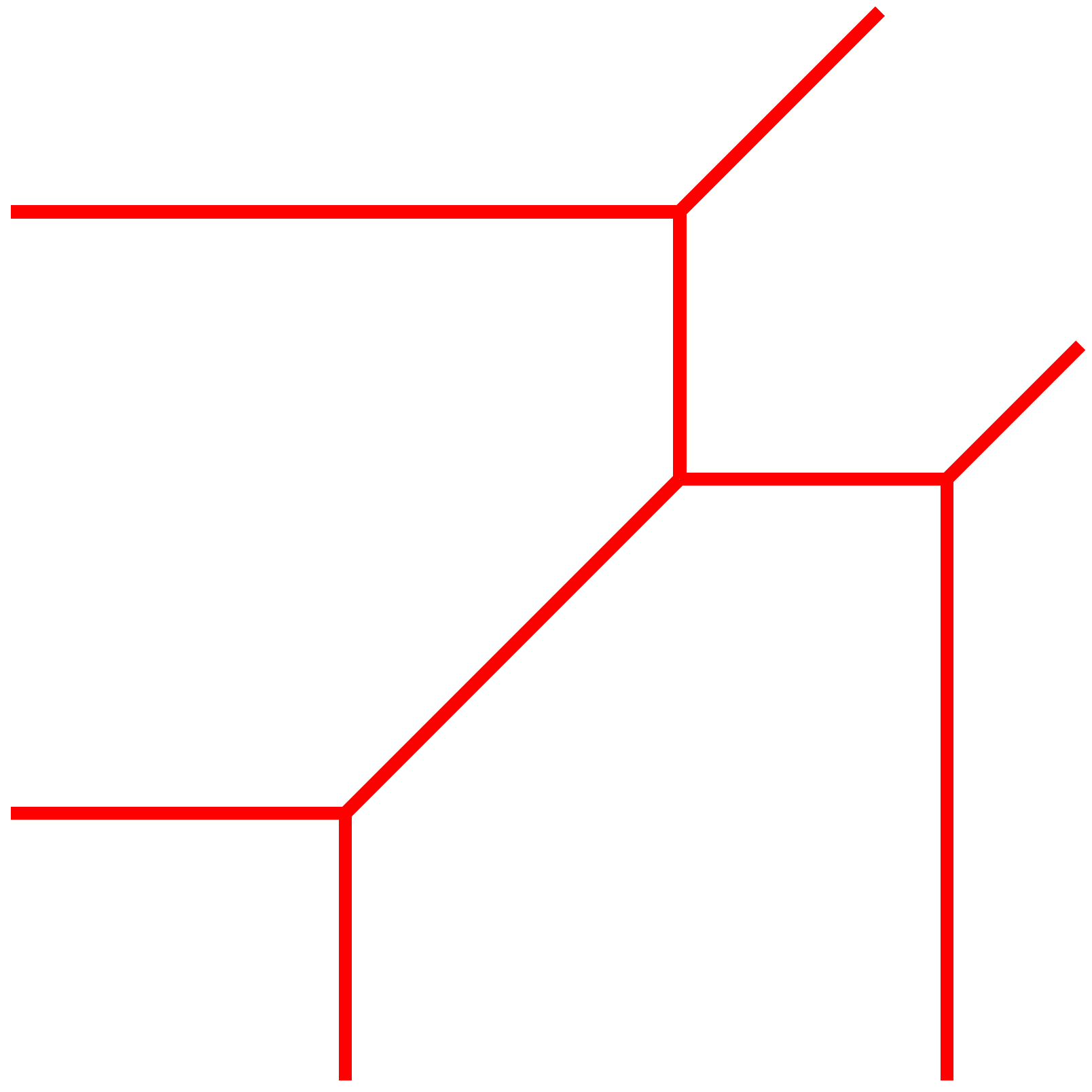}
&\includegraphics[width=3cm, angle=0]{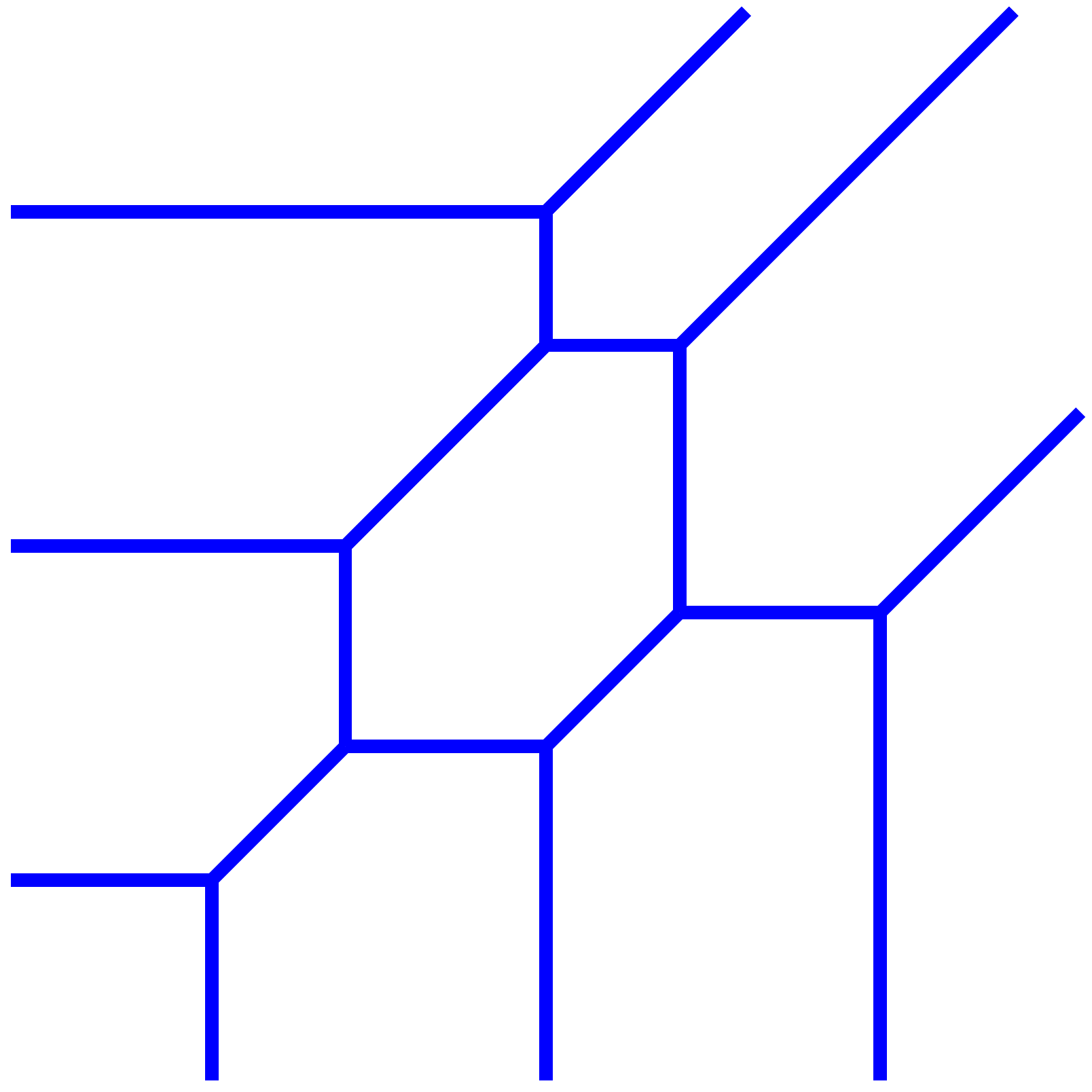}
&\includegraphics[width=3cm, angle=0]{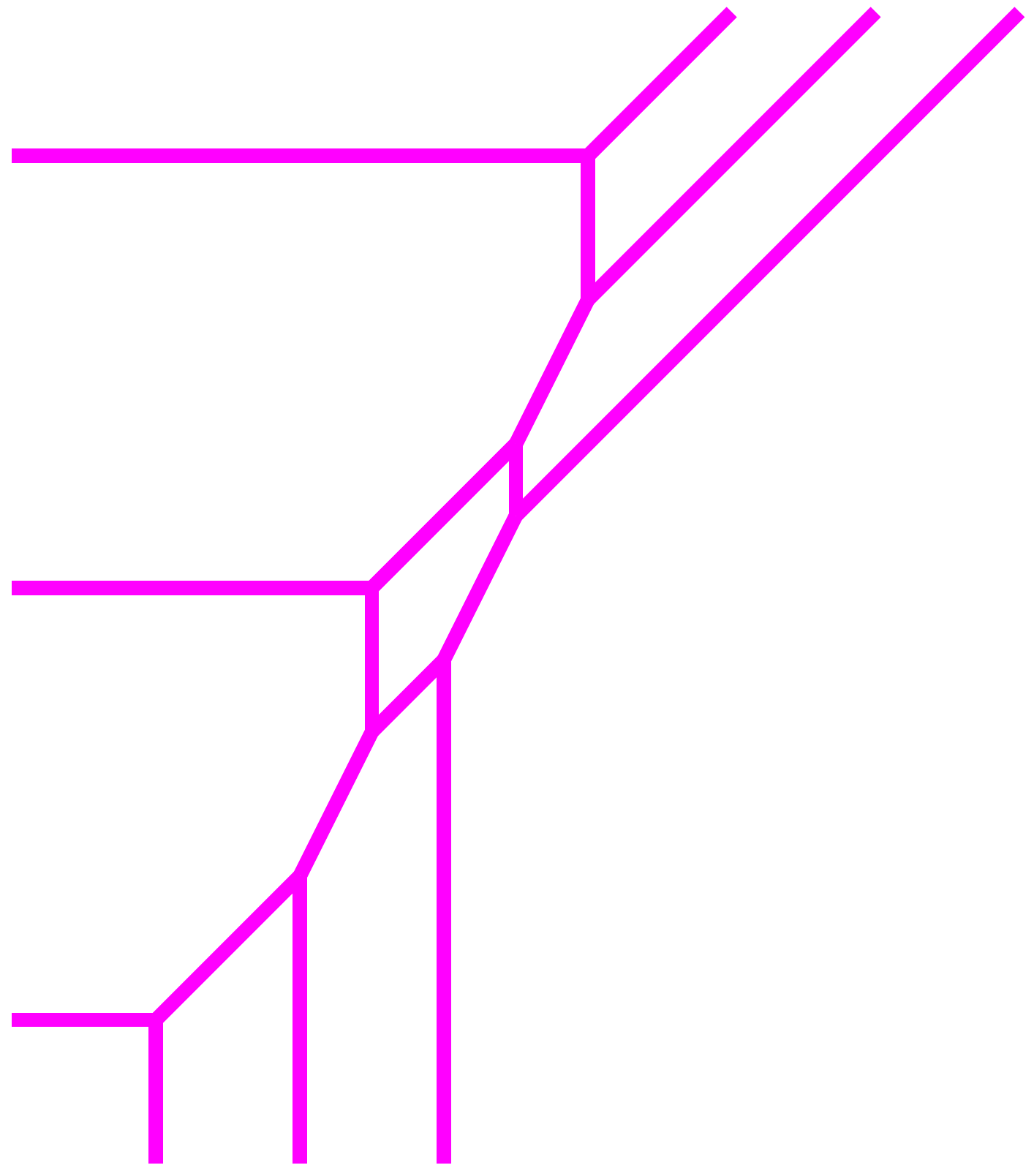}

\\
\\   A tropical line & A  tropical conic & A  tropical cubic &
Another  tropical cubic
\end{tabular}
\label{fig:balancing}
\caption{Examples of tropical curves in $\R^2$. All unbounded edges
have integer direction $(-1,0)$, $(0,-1)$, or $(1,1)$ toward infinity,
and all edges have weight $1$.}
\end{figure}
Tropical curves appeared in several mathematical and physical
contexts \cite{Aharony,Berg71,BiGr,V9,Mik1}, in particular in relation
with complex and non-Archimedean amoebas \cite{GKZ,Mik8,EKL06}. 

Figure \ref{fig:balancing} suggests a 
relation between the topology of
tropical curves  and of plane algebraic curves. 
Indeed by  \cite[Proposition 2.10]{Mik12}, the first
Betti number of a tropical curve in $\R^2$ of degree $d$ is at most
\[
\frac{(d-1)\cdot (d-2)}{2},
\]
and equality holds in the case of so-called \emph{non-singular}
tropical curves (i.e. tropical curves in $\R^2$ of degree $d$ with
exactly $d^2$ vertices). It is standard that the same is true
regarding  the geometric genus of an algebraic
curve of degree $d$ in the projective plane, see for example \cite[Chapter III 6.4]{Sha}. 
Such similarity led to use the expression 
``genus of a tropical curve'' in place of
``first Betti number of a tropical curve''.

Using linear projections, one easily sees that the above upper bound for the
geometric genus of an algebraic curve in the projective plane is also
an upper bound for the
geometric genus of an algebraic curve of degree $d$ in \emph{any}
projective space. The starting observation of this paper is that 
analogous statement does not hold in tropical geometry: 
there exist tropical curves of degree $d$ in $\R^n$, with $n\ge 3$, with
genus greater than the upper bound for tropical curves in
$\R^2$.
The first example is the tropical cubic curve of genus 2 in
$\R^3$ depicted in
Figure \ref{fig:cubicgenus2}.  
\begin{figure}[h!]
\begin{center}

  \includegraphics[height=9cm, angle=0]{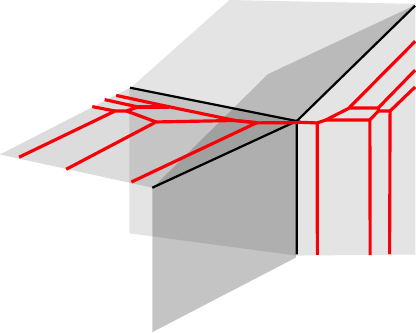}
      \put(-100,145){\textcolor{red}{$3$}}

\end{center}
\caption{A tropical cubic curve of genus 2 in $\R^3$. All unbounded edges
have integer direction $(-1,0,0)$, $(0,-1,0)$, $(0,0,-1)$, or $(1,1,1)$
toward infinity; all edges have weight 1, except the one  with weight
3 indicated close to it.} 

\label{fig:cubicgenus2}
\end{figure}
Moreover, this curve is
contained in a polyhedral complex $L$ of dimension 2: one vertex from
which emanate fours rays in the directions  $(-1,0,0)$, $(0,-1,0)$,
$(0,0,-1)$, and $(1,1,1)$, and six faces of dimension two generated by
each pair of rays. It turns out that $L$ is a \emph{tropical plane} in
$\R^3$, i.e. a tropical surface of degree 1  (see below).
Hence Figure \ref{fig:cubicgenus2} exhibits a rather surprising (to us) example
of a genus 2 tropical cubic in a tropical plane. We generalise this
observation in next Theorem, where a tropical curve in $\R^n$ is
called \emph{planar} if it is contained in a tropical plane.

\begin{thm}\label{thm:curve}
For any integers $d\ge 1$ and $n\ge 2$, there exists a planar tropical curve
of degree $d$ in $\R^n$ with genus
\[
(n-1)\cdot \frac{(d-1)\cdot (d-2)}{2}.
\]
\end{thm}
Generalising Figure \ref{fig:cubicgenus2}, there exists therefore a
planar tropical cubic curve of any given genus $g\ge0$.
Note that Theorem \ref{thm:curve} disproves in particular 
\cite[Conjecture 4.5]{Yu14}.

\medskip
T.~Yu proved in \cite[Proposition 4.1]{Yu14}
that a tropical curve of degree $d$ in $\R^n$
 has no more than $2 d^2\cdot(n-1)^2$ vertices, which implies 
that the
genus of such  tropical curve is bounded from
above by a constant depending only on $d$ and $n$. Nevertheless,
to our knowledge the following question remains open in general.
\begin{pb}
What is the maximal possible genus of a tropical curve of degree $d$
in $\R^n$?
\end{pb}
In the case of planar tropical curves in $\R^3$, we can
``almost''
prove that  Theorem~\ref{thm:curve} is optimal. 
\begin{thm}\label{thm:genus curve}
If $C\subset \R^3$ is a
planar tropical curve of degree $d$ with
$4d$ unbouded edges, then $C$ has genus at most $(d-1)\cdot(d-2)$.
\end{thm}
Note that a tropical curve $C$ of degree $d$ in $\R^3$ has at most $4d$
unbounded edges, and that there is equality if and only if $C$ has
exactly $d$ unbounded edges of weight 1 in each of the outgoing
direction
\[
(-1,0,0), \ (0,-1,0),\ (0,0,-1),\ (1,1,1).
\]
We believe that Theorem \ref{thm:genus curve} still holds without the
assumption on unbounded edges of $C$, and that a (quite technical)
adjustment of our proof should work. It is nevertheless not so clear
to us how to generalise our proof in higher dimensions.

\subsection{Higher dimensions}
Tropical curves generalise to tropical varieties in $\R^n$ of any
dimension. These are finite polyhedral complexes in $\R^n$ such that
all faces have a direction defined over $\Z$, all facets (i.e. faces
of maximal dimension) are equipped with a positive integer weight,
and which satisfy a
balancing condition at each face of codimension 1. 
We refer to \cite[Section 5]{BIMS15} for a precise
definition of tropical subvarieties of $\R^n$. 
By convention, a tropical variety will always be of pure dimension: 
every face  is contained  in 
a facet.

There is also a notion of degree of a tropical variety $X$ in $\R^n$,
based on  stable
intersections defined in \cite{RGJS05,Mik06}.
Recall that  a standard fan tropical linear space of dimension $k$ in $\R^n$ is a polyhedral
fan with a vertex from which emanate $n+1$ rays in the directions
\[
(-1,0,\cdots, 0), (0,-1,0,\cdots,0), \cdots,
(0,\cdots, 0,-1), (1,1,\cdots ,1),
\]
and having
$\binom{n+1}{l}$ additional faces of dimension $l\in\{2,\cdots,l \}$
generated by
each subset of $l$ of the $n+1$ rays.
The degree of a tropical variety $X$ of codimension $k$ in $\R^n$ is
defined
as the stable intersection number of $X$ with a generic standard fan tropical linear space of dimension $k$.

\medskip
The aim of this paper is to  study the topology of tropical
varieties. To this purpose, it is more convenient to deal with compact
tropical varieties, and to consider \emph{projective} tropical
varieties, i.e. tropical subvarieties of the tropical projective space
$\TP^n$.
This latter is defined as the quotient of
$([-\infty;+\infty[)^{n+1}\setminus \{(-\infty,\cdots,-\infty)\}$
by the equivalence relation
\[
(x_0,\cdots,x_n)\sim
(x_0+\lambda,\cdots,x_n+\lambda) \qquad \lambda\in \R,
\]
see for
example \cite[Section 3.3]{MikRau19}.
The tropical projective space $\TP^n$ is the union of finitely many
copies of $\R^k$ with $k\in\{0,\cdots,n\}$ defined by
\[
\R_I=\{[x_0:\cdots:x_n] \ | \  x_i=-\infty \mbox{ if and only if
}i\in I\}
\]
where $I\subsetneq \{0,\cdots,n\}$. A tropical variety in $\TP^n$ is the union of the topological closure of finitely
many tropical varieties contained in some $\R_I$.
The notion of degree of a tropical variety extends to projective
tropical varieties, see \cite[Section 5.2]{MikRau19}.

\medskip
Now we are ready to state the main problem studied in this paper, as
well as our main results.
We define the numbers 
$$B_j(\xdim,\xcod,d)= \sup_X \{b_j(X) \}\in\N\cup \{+\infty\},$$
where $X$ ranges over all tropical subvarieties of
dimension $m$ and degree $d$ in $\TP^{m+k}$.

\begin{pb}
Estimate the numbers $B_j(\xdim,\xcod,d)$.
\end{pb}

Generalising what we saw in the case of curves,
the values of the numbers $B_j(\xdim,1,d)$ are well known by
\cite[Proposition 2.10]{Mik12}: for $m=0$, $B_0(0,1,d)=d$ and for $m\ge 1$ and $d\ge 1$,
$$B_0(\xdim,1,d)=1, \quad
B_1(\xdim,1,d)=\cdots=B_{m-1}(\xdim,1,d)=0,\quad \mbox{and}\quad
B_m(\xdim,1,d)=\left(\begin{array}{c}d-1\\ m+1\end{array} \right).$$
This  follows from the existence of the dual subdivision of a tropical
hypersurface. Determining the exact value of $B_j(\xdim,k,d)$ for
$k>1$ seems more difficult, and it is even not clear a priori that this
number is finite. 
Our main result is the following.
\begin{thm}\label{thm:intro}
  Let $d,m$ and, $k$ be three positive integers.
Then the number $B_j(\xdim,\xcod,d)$ is finite for any $j$, and one has 
$$B_m(\xdim,\xcod,d)\ge \xcod\cdot B_m(\xdim,1,d).$$
\end{thm}

\begin{cor}\label{cor:unbounded}
  For any  integers $m\ge 1$ and $d\ge m+2$, we have
  $$\lim_{k\to+\infty} B_m(\xdim,\xcod,d)=+\infty.$$
\end{cor}

From our proof
that  $B_j(\xdim,\xcod,d)$ is finite, it is
possible to extract explicit upper bounds. Nevertheless these bounds seem
far from being sharp (for example we did not succeed 
 to obtain a
better upper bound than T.~Yu in the case of curves). The lower bound in Theorem \ref{thm:intro}
is obtained by constructing
explicit examples. To do so, we use a method of construction of
tropical varieties that we call floor composition (see
Section~\ref{Sec:composition}), and  which originates in the floor
decomposition technique introduced by Brugall\'e and Mikhalkin
(\cite{BruMikh07, BruMikh08,Br6}), and in the tropical modifications
introduced by Mikhalkin in \cite{Mik06}. It is worth noting that the floor
composed varieties we construct are
actually projective hypersurfaces, thus generalising 
 Theorem \ref{thm:curve}.
\begin{thm}\label{thm:intro 2}
  Let $d,m$ and $k$ be three positive integers.
  Then there exist a tropical
linear space $L$ of dimension $m+1$ in $\TP^{m+k}$, and 
a tropical hypersurface $X$ of degree $d$
in $L$ such that
$$b_m(X)\ge \xcod\cdot B_m(\xdim,1,d).$$ 
\end{thm}
In connection to algebraic geometry, it seems also  interesting to
determine the maximal value of Betti numbers of tropical hypersurfaces
of degree $d$ of a given tropical linear space.
At this time, we are not aware any generalisation 
 of
Theorem \ref{thm:genus curve} to tropical varieties of higher dimension.
\begin{rem}
All tropical varieties we construct in our proof of
Theorem \ref{thm:intro 2} are singular as soon as $k\ge 2$.
It may be interesting to
study bounds on Betti numbers, and more generally on tropical Hodge
numbers, of non-singular tropical projective varieties of a given
dimension, codimension, and degree. In particular, we do not know if
there exist  universal finite upper bounds which do not depend on
the codimension.
For example, it follows from the tropical adjunction
formula~\cite[Theorem 6]{Sha15} that the upper bound given by
Theorem~\ref{thm:genus curve} can be refined to the classical bound
$\frac{1}{2}(d-1)\cdot(d-2)$ under the additional assumption that $C$
is locally of degree 1 in $L$ (i.e. $C$ is a non-singular tropical
subvariety of $L$).
\end{rem}

\bigskip
Homology groups of a tropical variety $X$ are special instances of its tropical
homology groups (we refer to
\cite{MikZha14,BIMS15,KSW16} for the definition of  tropical
homology for locally finite polyhedral complexes in  $\TP^n$). More precisely, 
the group $H_j(X ;\R)$ is canonically isomorphic to  the tropical homology group $H_{0,j}(X
;\R)$.
Our proof of finiteness of the numbers $B_j(\xdim,k,d)$  in Theorem~\ref{thm:intro} also implies finiteness of the numbers
$$\sup_X
\{\dim H_{p,q}(X,\R) \}\in\N\cup \{+\infty\},$$
where $X$ ranges over all tropical subvarieties of
dimension $m$ and degree $d$ in $\TP^{m+k}$.
In the case of surfaces, we  compute  all
tropical homology groups of the tropical surfaces  constructed in the
proof of Theorem \ref{thm:intro 2}. Let us denote by $h_{p,q}^\C(d,m)$
the dimension of the $(p,q)$-tropical homology group of a non-singular
tropical hypersurface of degree $d$ in $\TP^{m+1}$.
By \cite[Corollary 2]{IKMZ},  this
number does not depend on a particular choice of a tropical
hypersurface, and  is equal to the $(p,q)$-Hodge
number of a non-singular complex algebraic hypersurface of degree $d$
in $\CP^{m+1}$. In particular we have
$$h^{\C}_{2,0}(d,2)=\frac{(d-1)\cdot(d-2)\cdot(d-3)}{6}
  \qquad\mbox{and}\qquad h^{\C}_{1,1}(d,2)=\frac{4d^3-12d^2+14d}{6}.$$
A tropical surface in $\TP^n$ is called \emph{spatial} if it is
contained in a tropical linear space $L$ of dimension $3$.
  
\begin{thm}\label{thm:main surface}
  Let $k$ and $d$ be two positive integers.
  Then there exist a 
spatial tropical surface $X$ of degree $d$
in $\TP^{2+k}$ with the following
tropical Hodge diamond
$$\begin{array}{ccccc}
  & & 1 & &
\\  & 0 & &\qquad 0  \qquad&
\\ \xcod\cdot h^{\C}_{2,0}(d,2) &
&
h^{\C}_{1,1}(d,2) +\frac{(k-1)\cdot(d-1)\cdot(2d^2-7d+9)}{3}
& & \xcod\cdot h^{\C}_{2,0}(d,2)
\\  & (k-1)\cdot(d-1) & & 0 &
\\  & &1 & &
\end{array} $$ 
 where we use the  convention  that $h_{0,0}$ is the topmost number and $h_{2,0}$ the leftmost one.
\end{thm}

Hence as soon as $d\ge 2$, the quantities $h_{1,1}(X)$ and
$h_{2,1}(X)$ are not bounded
from above among spatial tropical surfaces of degree $d$.
Our proof of Theorem
\ref{thm:main surface}
generalises the computation by K.~Shaw of
tropical homology groups of floor
composed surfaces in $\TP^3$ \cite{Sha13}. 
We point out that the technique developed to prove Theorem
\ref{thm:main surface} also applies to study tropical Hodge numbers of  floor composed tropical varieties of any dimension. Nevertheless computations become a bit tedious starting from dimension 3, so we restricted ourselves to the case of surfaces.

\subsection{Comparison with algebraic geometry}
To a great extent, the tremendous development of tropical geometry the
last fifteen years has been motivated by its deep relations to
algebraic geometry. There exists several procedures
that associate a tropical variety $X$ to a family of projective
complex algebraic varieties $(\mathcal X_t)$. For such a
\emph{realisable} tropical variety, the tropical Hodge numbers may be
bounded from above in terms of
the Hodge numbers of a general member of the family  $(\mathcal X_t)$,
see for
example \cite[Corollary 5.8]{HelKat12}, \cite[Corollary 5.3]{Katz2}, and
\cite[Corollary 2]{IKMZ}.

Hence it is reasonable to compare our
main results stated above to what is known about Hodge numbers of projective
complex algebraic varieties.
As usual, in the case of hypersurfaces (and more generally  of
complete intersections) in $\TP^n$, both series of geometric invariants coincide:
it follows from \cite[Corollary 2]{IKMZ} that the tropical Hodge numbers of a
non-singular tropical hypersurface equal the Hodge numbers of
a non-singular complex algebraic hypersurface of the same
dimension and degree\footnote{Note however that this correspondence only concerns
  dimension of the corresponding vector spaces.
  There is no canonical isomorphism
  between tropical homology groups and Hodge groups in general.}.

Given two positive integers $m$ and $d$, the Hodge number $h^{p,q}(\mathcal
X)$ of a projective complex algebraic variety $\mathcal X$ of degree
$d$ and dimension $m$ is bounded
from above by some constant that only depends on $m$ and $d$, see
\cite{Mil64,Har81}. For example, it is well known that a
cubic curve in $\CP^n$ has genus at most 1, whatever the value of $n$
is. Corollary \ref{cor:unbounded} and Figure \ref{fig:cubicgenus2}
show that the situation is
drastically different in tropical geometry, where such an upper
bound independent on the codimension does not exist. In particular,
for 
$k$ large enough with respect to some fixed $m$ and $d$, the tropical
varieties whose existence is 
attested by Theorem \ref{thm:intro 2} are not the tropicalisation of any family of projective varieties of the same dimension and degree.

\medskip
In a somewhat similar direction, 
Davidow  and Grigoriev studied in  \cite{DavGri17} the possible numbers of
connected components of intersections of tropical varieties.
They proved in particular that this number can also be much larger than the bound in algebraic geometry given by Bézout Theorem.

\bigskip
\noindent {\bf Organisation of the paper.}
Section~\ref{Sec:upper} is devoted to showing the finiteness of $B_j(\xdim,\xcod,d)$ 
  and proving Theorem \ref{thm:genus curve}. 
  In Section~\ref{Sec:composition} the floor composition method is
  introduced and we explain how to compute Betti numbers of the
  obtained varieties.
  In Section~\ref{Sec:lower}, we first prove Theorem~\ref{thm:main curve} which contains Theorem~\ref{thm:curve}. 
 We then give lower estimates of $B_j(\xdim,\xcod,d)$ in general
 using Theorem~\ref{thm:main curve} as induction basis,
 and floor composition to recursively construct the varieties of Theorem~\ref{thm:intro 2}.
 Section~\ref{sec:trop hom} is dedicated to the computation of tropical
 Hodge numbers of floor composed tropical surfaces and to the proof of Theorem ~\ref{thm:main surface}.

\medskip
\noindent {\bf Acknowledgement.} 
We are grateful to Kristin Shaw for her disponibility to explain many
aspects of her previous works, and more generally for many enlightening
discussions about tropical homology. We also thank an anonymous
referee for many useful remarks about a first version of this paper.

This research has been supported by ECOS NORD M14M03, UMI 2001, 
Laboratorio Solomon Lefschetz CNRS-CONACYT-UNAM, Mexico. L. L.d.M. 
was also supported for this research by PAPIIT-IN114117 and PAPIIT- IN108216. 
 Part of this work has been achieved during visits of B.B. and E.B. at
Universidad Nacional Autónoma de México (Instituo de Matemáticas, Unidad Cuernavaca), and of B.B. and L. L.d.M. at Centre Mathématiques Laurent Schwartz. We
thank these institutions for their support and the excellent working
conditions they offered us.

 \section{Upper estimates}\label{Sec:upper}

In this section, we prove the finiteness of the numbers
$B_j(\xdim,\xcod,d)$, and 
Theorem \ref{thm:genus curve}.
The main ingredient is  tropical intersection theory, for which we refer to \cite{AlRa1,Sha13-2,Br17} for more details.
  
\subsection{Finiteness of $B_j(\xdim,\xcod,d)$} 
 Our strategy to prove the finiteness of $B_j(m,k,d)$
is to reduce to the case of hypersurfaces by a
suitable projection  universal for all tropical subvarieties  of
dimension $m$ and degree $d$ in $\TP^{m+k}$. 
We denote by $Gr(m,\Z^{m+k})\subset Gr(m,\R^{m+k})$ the space of
subvector spaces of dimension $m$ of $\R^{m+k}$ that are defined over $\Z$.
\begin{lemma}\label{lem:finite dir}
  Let  $\mathcal V(d,m,k)$ be the set of elements of $Gr(m,\Z^{m+k})$
  that are the direction of a facet of a tropical variety
  of dimension $m$ and degree $d$ in $\TP^{m+k}$. Then
  $\mathcal  V(d,m,k)$ is a finite set.
\end{lemma}
\begin{proof}
  The usual Plücker embedding of $Gr(m,\R^{m+k})$ lifts to an injection
  $$\begin{array}{cccc}
   \phi:& Gr(m,\Z^{m+k})&\longrightarrow&
   \Lambda^m\left(\Z^{m+k}\right)/\{\pm 1\}
   \\ & \mbox{Span}(v_1,\cdots, v_m)&\longmapsto& v_1\wedge \cdots\wedge v_m
 \end{array},$$ 
 where $(v_1,\cdots, v_m)\in (\Z^{m+k})^m$ is a basis of the lattice
 $\mbox{Span}(v_1,\cdots, v_m)\cap \Z^{m+k}$.  In the standard
 coordinates of $\Lambda^m\left(\Z^{m+k}\right)$, the coordinates of
 $\phi\left(\mbox{Span}(v_1,\cdots, v_m)\right)$ are given by all $m\times m$
 minors of the matrix $(v_1,\cdots, v_m)$.

 Suppose now that $V\in \mathcal  V(d,m,k)$, and choose a basis 
 $(v_1,\cdots, v_m) $ of $V\cap \Z^{m+k}$.
 Let $(u_1,\cdots,u_{m+k})$ denote the canonical basis of $\R^{m+k}$. 
 Then by the tropical
  Bézout Theorem, one has
  $$\left|\det(u_{i_1},\cdots ,u_{i_k},v_1,\cdots,v_m) \right|\le d $$
  for any subset $\{i_1,\cdots,i_k\}\subset\{1,\cdots,m+k\}$. 
 All these determinants are precisely the $m\times m$  minors of the matrix $(v_1,\cdots, v_m)$. Hence we deduce that $\phi\left( \mathcal  V(d,m,k)\right)$ is a finite set, and so is $\mathcal  V(d,m,k)$.
\end{proof}

\begin{prop}\label{prop:finite cells}
 Given any  integers $m,k,d>0$ and $j\ge 0$, the number  $B_j(\xdim,\xcod,d)$ is finite.
  \end{prop}
\begin{proof}
  Without loss of generality, we only consider tropical
  subvarieties of $\TP^n$ with no irreducible component contained in
  $\TP^n\setminus\R^n$.
  Then  the number of faces of dimension $j$ of a tropical
  hypersurface
  in $\R^{n}$ is equal to the
  number of faces of dimension $n-j$ in its dual subdivision.
  Any tropical hypersurface of degree $d$ in $\TP^n$ has a Newton
  polytope included in
the simplex with vertices 
$$(0,\cdots,0), \ (d,0,0,\cdots,0), \ (0,d,0,\cdots,0),\ \cdots,
\ (0,\cdots,0,d,0), \ (0,\cdots,0, 0,d).$$ 
Hence the proposition holds true in the case of hypersurfaces,
i.e. when $k=1$.

\medskip
In the case when $k\ge 2$, we prove the proposition by induction on
$m$, the case $m=0$ holding trivially.
Note that a tropical subvariety of $\TP^n$
carries a canonical polyhedral
decomposition when it is either a curve or a hypersurface, however
this is
 no longer 
the case in higher dimensions and codimensions
(think for example of the union of the two $2$-planes with equations $x_1=x_2=0$
and $x_3=x_4=0$ in $\R^4$).
Given any couple $(V,V')$ of distinct
elements of  $\mathcal  V(d,m,k)$, we fix a vector $u_{V,V'}\in
V'\setminus V$, and we define
$$\mathcal  W= \{V\oplus \R u_{V,V'}\  | \ (V,V')\in \mathcal  V^2(d,m,k)
\mbox{ and } V\ne V' \}.$$
By  Lemma \ref{lem:finite dir},
both sets $\mathcal  V(d,m,k)$ and $\mathcal  W$ are finite, and there exists a vector space $W\in Gr(k-1,\Z^{m+k})$ such that
$$W\cap V=\{0\} \qquad \forall V\in  \mathcal  V(d,m,k)\cup \mathcal  W.$$
Let $\pi: \TP^{\xdim+\xcod} \to \TP^{\xdim+1}$ be the tropical map
induced by the linear projection
along $W$ in $\R^{m+k}$. Note that for any tropical variety $X$
of dimension $m$ and degree $d$ in $\TP^{m+k}$,
the degree  of $\pi(X)$ in $\TP^{m+1}$
is bounded from above by (and generically is equal to) a constant $D(d,W)$
that only depends on $d$ and $W$.
Since $W\cap V=\{0\}$ for any $V\in \mathcal  V(d,m,k)$, 
the dimension of $\pi(F)$ equals the one of
 $F$ for any facet $F$ of $X$.
The condition that $W\cap V=\{0\}$ for any $V\in \mathcal W$ guaranties
that different elements of $ \mathcal  V(d,m,k)$ have distinct images
under $\pi$.
From now on, we consider the lift to $X$ of 
the canonical polyhedral decomposition of $\pi(X)$.
By construction, 
the preimage $\pi^{-1}(F)$ of any open facet $F$ of $\pi(X)$
is the disjoint union of open facets of $X$.

Let $W'$ be any element of $Gr(k,\Z^{m+k})$ which contains $W$ and
such that $W'\cap V=\{0\}$ for any $V\in \mathcal  V(d,m,k)$. The 
tropical map $\pi':X \to \TP^{\xdim}$ induced 
by the linear projection along $W'$ is finite. Furthermore the
tropical degree  of $\pi'$
is bounded from above by (and generically is equal to) a constant
$D'(d,W')$
that
only depends on $d$ and $W'$.
By the tropical Bézout Theorem,   the
fibre $\pi'^{-1}(x)$ contains  at most $D'(d,W')$ points for any point $x$ in $\TP^m$. Since each fibre of $\pi_{|X}$ is contained in a fibre of $\pi'$, there are  at most $D'(d,W')$ points in $\pi^{-1}_{|X}(x)$ for any $x\in\pi(X)$.
 Hence, we deduce that the number of facets  of $X$, and so the number
 $B_m(\xdim,\xcod,d)$, is at most
 $$ D'(d,W')\cdot K(\xdim,D(d,W))<+\infty,$$
 where $K(\xdim,D(d,W))$ is the maximal number of facets of a tropical hypersurface of degree $D(d,W)$ in $\R^{m+1}$. This proves the proposition when $j=m$. 

 \medskip
 We prove the results for $j<m$ by induction on $m$. Let us denote respectively by $Sk^{m-1}(X)$ and $Sk^{m-1}(\pi(X))$ the  closure in $\TP^{m+k}$ and $\TP^{m+1}$ of the $(m-1)$-skeleton of $X\cap\R^{m+k}$  and $\pi(X)\cap\R^{m+1}$. 
 The stable self-intersection of $\pi(X)\cap\R^{m+1}$ in $\R^{m+1}$
 provides positive integer weights on $Sk^{m-1}(\pi(X))$,  turning
 this latter into a tropical subvariety of  $\TP^{m+1}$ of
 degree at most $D(d,W)^2$. In its turn, the stable intersection of
 $X\cap\R^{m+k}$ with  $\pi^{-1}(X)\cap\R^{m+k}$ provides positive integer weights on $Sk^{m-1}(X)$, turning it into a tropical subvariety of  $\TP^{m+k}$  whose degree is bounded by a number $\widetilde D(d,W)$ which only
 depends on $d$ and $W$. Since $X$ is obtained from $Sk^{m-1}(X)$ by
 attaching $m$-cells, we have
 $$B_j(\xdim,\xcod,d)\le B_j(\xdim-1,k+1,\widetilde D(d,W)).$$
 Since by assumption the number $B_j(\xdim-1,k+1,\widetilde D(d,W))$
 is finite, the proposition is proved.
\end{proof}

\subsection{Auxiliary statements}
The proof of Theorem~\ref{thm:genus curve} requires the following
several auxiliary lemmas that will be combined in Section
\ref{sec:proof genus curve}.
Given a tropical plane $L$ in $\TP^3$, we denote by
$Sk^i(L)$ the closure
in $\TP^3$
of the union of all faces of dimension $i$ of  $L\cap \R^3$.

Until the end of this section, we denote by $L_0$ the
tropical plane in $\TP^3$
defined by the tropical polynomial $\tg x+y+z+0\td$.
\begin{lemma}\label{lem:vertex H}
  Theorem \ref{thm:genus curve} holds if every  edge  of $C$ is either disjoint from $Sk^1(L)\setminus Sk^0(L)$ or contained in $Sk^1(L)$.
\end{lemma}
\begin{proof}
Denote by $C_F$ the intersection of $C$ with a facet $F$ of $L$.  By
assumption,  one has
$$b_1(C)=\sum_F b_1(C_F),$$
where the sum runs all over the facets of $L$.
Let $\sigma$ be the number of the directions  $(0,0,-1)$, $(0,-1,0)$,
$(-1,0,0)$, and $(1,1,1)$ along which $L$ is not a cylinder, and let
$s$ be one of these four directions.
The projection along  $s$ defines a  degree one tropical map $\pi_s:L\to\TP^2$, and by assumption, one has
 $$\sum_F b_1(C_F) = b_1(\pi_s(C))\le \frac{(d-1)\cdot (d-2)}{2},$$
 where the sum runs over all facets of  $L$ not containing the direction $s$.
 Considering all
 possible directions $s$, each $C_F$ contributes to the first
Betti number of exactly two projections $\pi_s(C)$, so we get
$$2g(C)\le \sigma \cdot \frac{(d-1)\cdot(d-2)}{2}\le 4\cdot \frac{(d-1)\cdot(d-2)}{2}, $$
which is the desired result.
\end{proof}

Given a tropical curve $C$ in $\TP^\damb$ and $p\in C$, we denote by
$\val_p(C)$ the valency of $C$ at $p$, and by $\Ed_p(C)$ the set of edges
of $C$ adjacent to $p$ (viewed as a vertex of $C$). We also
define $C^0=C\cap \R^n$ and 
$C^{\infty}=C\cap \left( \TP^n\setminus\R^n \right)$.
If $C$ is furthermore contained in a tropical plane $L$, we denote by
$\Ed^2(C)$ the set of edges of $C$ that are not contained in $Sk^1(L)$.
 The tropical curve $C$ is called a
\emph{fan} tropical curve with vertex $v$ if the  support of $C$
is the closure in $\TP^3$ of rays in $\R^3$ all emanating from $v$.

 \smallskip
The proofs of next Lemmas and of Theorem \ref{thm:genus curve}
extensively use tropical intersection theory of tropical curves in
tropical surfaces, for which 
we use  the presentation given in
\cite[Section 3]{Br17} and \cite[Section 6.2]{BIMS15}. For the reader
convenience, we recall informally the definition of the local
self-intersection $C^2_p$ at a point $p$ of a tropical curve $C$ in $L_0$
with no irredducible components in $\TP^3\setminus \R^3$.
Recall that each facet $F$ of $L_0$ contains a unique corner point
(i.e.  with coordinates $(-\infty,-\infty)$ in an affine tropical chart
$[-\infty;+\infty[^2$) that we
denote by $q_F$. 
There are several cases to consider to define $C^2_p$, depending on the location of the
point $p$ in $L_0$:
\begin{itemize}
\item  $p$ is contained in $C^0\setminus Sk^1(L_0)$:  the tropical plane $L_0$ is then
  locally given at $p$ by the affine tropical chart $\R^2$, and 
  $C^2_p$ is
  defined as the stable intersection of $C$ at $p$ in $L_0$ \cite{RGJS05};

  \item $p$ is contained in $Sk^1(L_0)\setminus Sk^0(L_0)$:
    the curve $C$ can be deformed in $L_0$ locally at $p$ into a tropical curve $\widetilde C$
    intersecting $C$ in $L_0\setminus Sk^1(L_0)$, see Figure
    \ref{fig:selfinter}; we define $C^2_p$ as the sum of the
    tropical intersection multiplicity of intersection points of $C$
    and $\widetilde C$ that are close to $p$ (depicted in black dot
    points in Figure \ref{fig:selfinter}b,  note that $C^2_p=0$ if
    $p\in C^\infty$);
   \begin{figure}[h!]
\begin{center}
\begin{tabular}{ccc}
  \includegraphics[height=4cm, angle=0]{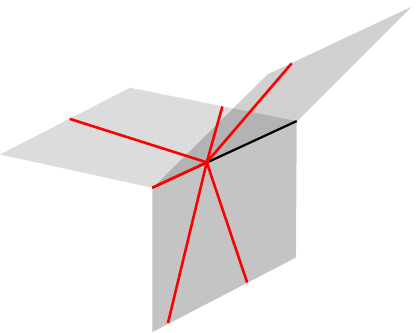}
  \put(-130,75){$C$}
&\includegraphics[height=4cm, angle=0]{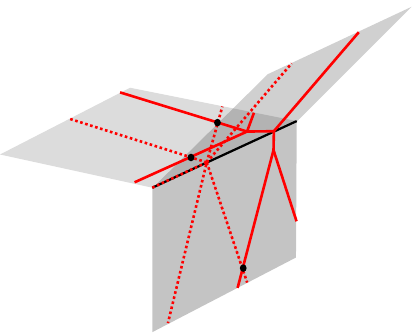}
  \put(-130,75){$C$}
  \put(-115,85){$\widetilde C$}
&\includegraphics[height=4cm, angle=0]{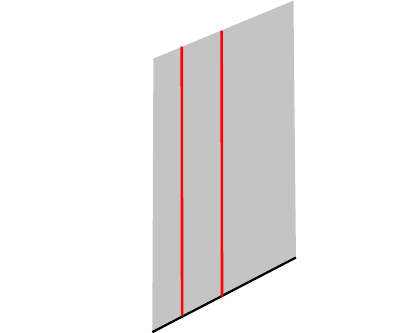}
  \put(-60,90){$\widetilde C$}
  \put(-90,100){$C$}

\\
\\ a) $p\in C^0\cap Sk^1(L_0)\setminus Sk^0(L_0)$  & b)
&  c) $p\in C^\infty$
\end{tabular}
\end{center}
\caption{}
\label{fig:selfinter}
\end{figure}

  \item $p=q_F$ for some facet $F$ of $L_0$: in an affine tropical
    chart of $L_0$ at  $q_F$, the tropical curve $C$ is defined
        by a tropical polynomial $P(x,y)$; denoting by $\Delta_{q_F}(C)$ the
        Newton polygon of $P(x,y)$, by $\overline\Delta_{q_F}(C)$ the convex
        hull of $\Delta_{q_F}(C)\cup\{(0,0)\}$, and by
        $\Gamma_{q_F}(C)=\overline\Delta_{q_F}(C)\setminus \Delta_{q_F}(C)$, we define
        \[
        C^2_{q_F}=Area (\Gamma_{q_F}(C)),
        \]
        where $Area$ stands for twice the Euclidean area;

      \item
        $p$ is the origin in $\R^3$: denote by $\widetilde C$ the fan
        tropical curve that coincide with $C$ in a neighborhood of
        $p$, and  by $d$ the degree of $\widetilde C$; we define
        \[
        C^2_p=d^2 -\sum_{F\mbox{ facet of }L_0}  \widetilde C^2_{q_F}.
        \]
\end{itemize}

Note that we have $C^2_p\ge 0$ whenever  $p$ is not the origin. Next lemma
provides a lower bound for $C^2_0$. Recall that each edge $e$ of
a tropical curve is equipped with a weight $w_e\in \Z_{>0}$.

\begin{lemma}\label{lem:vertex H2}
  Let $C\subset L_0$ be a fan tropical curve of degree $d$ with vertex the origin. Then one has
  $$C_{0}^2+\sum_{e\in \Ed^2(C)} (w_e-1) -\val_{0}(C)\ge -d^2 +2d-4. $$
\end{lemma}
\begin{proof}
  Let $\widetilde C$ be a perturbation of $C$ outside a neighbourhood of
the origin into a tropical
curve of degree $d$ such that
\begin{itemize}
\item $\widetilde C$ is still contained in $L_0$;
\item any  vertex $v$ of $\widetilde C^0$ distinct from and not
adjacent to the origin (resp. connected to the origin by
an edge $e$) is trivalent and
$$\widetilde C^2_v= 1 \qquad (\mbox{resp.}\ \widetilde C^2_v= w_e).$$
\item $\widetilde C\cap\R^3$ has unbounded edges only in the standard
  directions $(0,0,-1)$, $(0,-1,0)$,
$(-1,0,0)$, and $(1,1,1)$.
\end{itemize}
Such  perturbation $\widetilde C$ exists: it suffices to
perturb $C$ in a neighborhood of each point $q_F$ according to any
convex triangulation of $\Gamma_{q_F}(C)$ such that each edge of
$\Gamma_{q_F}(C)\cap \Delta_{q_F}(C)$ is the edge of a triangle, and
 containing the maximal number of triangles among all triangulations satisfying this condition.
An elementary Euler characteristic computation gives
\begin{equation}\label{eq:euler}
  2 g(\widetilde C)=\sum_{v\in \widetilde C^0}(\val_v(\widetilde C)-2) +2 -\left|\widetilde C^\infty\right|.
  \end{equation}
Next, by \cite[Definition 3.6]{Br17} we have
\begin{align*}
  C^2_{0}&=d^2- \sum_{v\in \widetilde C^0\setminus \{0\}} \widetilde{C}^2_{v}
  \\ &=d^2- \sum_{v\in \widetilde C^0\setminus \{0\}}(\val_v(\widetilde C)-2)   -\sum_{e\in \Ed^2(C)} (w_e-1),
\end{align*}
from which we deduce that
$$\sum_{v\in\widetilde C^0}(\val_v(\widetilde C)-2) = d^2- C^2_{0} -\sum_{e\in \Ed^2(C)} (w_e-1)+\val_{0}(C)-2.$$
Just as $C$, the tropical curve $\widetilde C$  satisfies to the
hypothesis of Lemma \ref{lem:vertex H}. Hence
combining this latter identity together with  Lemma \ref{lem:vertex H}
and equation $(\ref{eq:euler})$, we obtain
$$d^2- C^2_{0}
-\sum_{e\in \Ed^2(C)} (w_e-1)+\val_{0}(C)- |\widetilde C^\infty| \le 2d^2-6d+4. $$
Now the result follows from the fact that $|\widetilde C^\infty|\le 4d$.
\end{proof}

\begin{lemma}\label{lem:vertex 1sk}
Let $L\subset\TP^3$ be a  tropical plane, and let
$C\subset L$ be a 
fan tropical 
curve with vertex $v_0$ contained in $Sk^1(L)\setminus Sk^0(L)$.  Then one has
$$C_{v_0}^2 \ge \val_{v_0}(C) -3. $$
\end{lemma}
\begin{proof}
  The lemma is true if $v_0$ is a trivalent vertex of $C$ since in
  this case $C_{v_0}^2=0$. If $v_0$ is not a trivalent vertex of $C$,
  then  we  perturb $C$
as depicted in Figure \ref{fig:selfinter}a and b,
  into a tropical curve $\widetilde C$ such
  that
\begin{itemize}
  \item $\widetilde C$ is contained in $L$;
  \item $C$ and $\widetilde C$ have the same directions of unbounded edges;
    \item $\widetilde C$ intersects $Sk^1(L)$ in a
      single trivalent vertex $\widetilde v_0$.
\end{itemize}
We have
$$C_{v_0}^2=\sum_{v\in \widetilde C^0}\widetilde C_v^2.$$
Furthermore if $v\ne v_0$ and $v$ is not a
2-valent point of $\widetilde C$, it follows from Pick Formula that 
$$\widetilde C_v^2\ge \sum_{e\in\Ed_v(\widetilde C)}w_e -2.$$
 Furthermore we have 
$$\widetilde C_{\widetilde v_0}^2 =0=\val_{\widetilde v_0}(\widetilde C) -3,$$
 and the result follows.
\end{proof}

 \subsection{Proof of Theorem \ref{thm:genus curve}}\label{sec:proof genus curve}
 Let us denote by $L\subset \TP^3$ a tropical plane containing $C$,
 and by $a$  the number of vertices of $C$ which are contained
 in $Sk^1(L)\setminus Sk^0(L)$.
 We claim that 
\begin{equation}\label{equ:ineq val}
\sum_{v\in C^0}\left(\val_v(C) -2 \right)\le 2d^2 -2d+ 2.
 \end{equation} 
Assuming that this inequality holds,
we have 
\begin{align*}
  2g(C)&=\sum_{v\in C^0}(\val_v(C)-2) +2- | C^\infty| \\
  &\le 2d^2 -6d+4 \\
  &\le 2(d-1)\cdot(d-2).
  \end{align*}

Hence it remains to prove Inequality $(\ref{equ:ineq val})$.
Suppose first that $C$ does not pass through $Sk^0(L)$. 

The
self-intersection of $C$ in $L$ is equal to $d^2$,
hence it follows from Pick Formula and Lemma \ref{lem:vertex 1sk} that
\begin{align*}
  d^2&=\sum_{v\in C^0} C^2_v 
  \\ &\ge \sum_{v\in C\setminus Sk^1(L)}\left(\sum_{e\in\Ed_v(C)}w_e
  -2 \right)   + \sum_{v\in C^0\cap Sk^1(L)}\left(\val_{v}(C) -3 \right) 
  \\ \\ & \ge \sum_{v\in C^0}\left(\val_v(C) -2 \right) -a.
  \end{align*}
Since $a \le d$, we have $d^2+a\le 2d^2 -2d+ 2$ and Inequality
$(\ref{equ:ineq val})$ holds.

Suppose now that $C$  passes through  $Sk^0(L)$. 
Again, it follows from Pick Formula and Lemma \ref{lem:vertex 1sk} that
\begin{align*}
  d^2&=\sum_{v\in C^0\setminus Sk^1(L)} C^2_v +
  \sum_{v\in C^0\cap Sk^1(L)\setminus\{0\}} C^2_v + C^2_{0} 
  \\ \\ &\ge \sum_{v\in C^0\setminus Sk^1(L)}
  \left(\sum_{e\in\Ed_v(C)}w_e -2 \right)   +
 \sum_{v\in C^0\cap Sk^1(L)\setminus\{0\}}\left(\val_{v}(C) -3 \right)  +
  C^2_{0}
  \\ \\ & \ge \sum_{v\in C^0\setminus  \{0\}}
  \left(\val_v(C) -2 \right) -a + \sum_{e\in \Ed^2_{0}(C)}
  (w_e-1)+   C^2_{0}
 \\ \\ &\ge \sum_{v\in C^0}\left(\val_v(C) -2 \right)-a + \sum_{e\in \Ed^2_{0}(C)}
  (w_e-1) +   C^2_{0} -\val_{0}(C)+2.
  \end{align*}
Denoting by $d_0$ the local intersection number of
 $C$ and $Sk^1(L)$ at the origin,
it follows from Lemma \ref{lem:vertex H2} that
$$\sum_{v\in C^0}\left(\val_v(C) -2 \right)\le  d^2
+ d_0^2-2d_0+2 +a.$$
Since the total
intersection number of $C$ and $Sk^1(L)$
is equal to $d$, and that each local intersection multiplicity on
$Sk^1(L)\setminus \{0\}$ 
is
positive, we deduce that $d_0\le d-a$ and $a\le d-1$. In particular we
have
$d_0^2-2d_0\le (d-a)^2-2(d-a)$, and 
\begin{align*}
  \sum_{v\in C^0}\left(\val_v(C) -2 \right) &\le 2d^2 -2d+ 2  +a\cdot (a-2d
  +3)
  \\ &\le 2d^2 -2d+
  2, 
\end{align*}
i.e. Inequality
$(\ref{equ:ineq val})$ holds in this case as well.
\hfill \BasicTree[1.3]{orange!97!black}{green!90!white}{green!50!white}{leaf}

 \section{Floor composition}\label{Sec:composition}
We describe a method of construction of tropical varieties which we
will use in Section~\ref{Sec:lower} to exhibit tropical varieties with
large top Betti numbers. This method originates in the floor
decomposition technique introduced by Brugall\'e and Mikhalkin
(\cite{BruMikh07, BruMikh08,Br6}), whose roots can in their turn be traced back to 
earlier ideas by Mikhalkin. A floor composed tropical variety is
a $\xdim$-dimensional tropical variety in $\R^{\damb+1}$ which is built out of
the data of a collection of $\xdim$-dimensional varieties in $\R^\damb$
together with some effective divisors on elements of this collection. 
In the case when the varieties and the divisors involved are homology 
bouquets of spheres, we express the Betti numbers of the floor composed 
 variety in term of  those of the construction's data.  

\subsection{Tropical birational modifications}
Here we slightly generalise the notion of
\emph{tropical modifications} introduced in \cite{Mik06} and further
developped in \cite{BrLop,Sha13,Sha13-2,Br18,CuetoMarkwig}.
Let $X$ be a tropical subvariety in $\R^n$. Recall (see for example
\cite[Section 5.6]{BIMS15} or \cite[Section 4.4]{MikRau19}) that to 
a tropical rational function $f:X\to\R$ corresponds its divisor $\div_X (f)$ which is a codimension 
one tropical cycle on $X$. 
We denote $\Gamma_f(X)\subset X\times\R\subset \R^n\times\R$ the graph of $f$ with weights inherited from $X$. 
Given  a closed polyhedron  $F$ in $\R^n\times\R$  equipped with a weight $w_F$, we denote by 
 $F^-$ (resp. $F^+$) the polyhedral cell $F - \R_{\ge 0}(0,\cdots,0,1)$ (resp. $F + \R_{\ge 0}(0,\cdots,0,1)$) equipped with the weight $w_F$. 

\begin{defi}\label{Def:Modification}
  Let $X$ be a tropical variety in $\R^n$, and  $f:X\to\R$
  be a tropical rational function. Suppose that there exist
   two effective tropical divisors $D_+$ and $D_-$ on $X$ such that
   $\div_X (f)  =D_+- D_-$. 
   The tropical variety $\widetilde X$ in $\R^{n+1}$ defined by
   $$\widetilde X=\Gamma_f(X)\ \bigcup \ \Gamma_f(D_+)^-\
   \bigcup\ \Gamma_f(D_-)^+$$ 
   is called a \emph{birational tropical modification} of $X$ along the
   divisor $D_+- D_-$. If $D_-=  \emptyset$, then $\widetilde X$ is called a \emph{tropical modification} of $X$ 
   along  $D_+$. 
\end{defi}

Our definition of tropical modification coincides with the definition 
from \cite{Mik06,Sha13,Sha13-2,BIMS15}.

\begin{exa}
The tropical line $L$ in $\R^2$ defined by the tropical polynomial
$\tg x+y+0 \td$ is a  tropical modification of $\R$ along $0$.
The tropical plane in $\R^3$ defined by the tropical polynomial
$\tg x+y+z+0 \td$ is a  tropical modification of $\R^2$ along the line $L$.
More generally, any  tropical linear space of dimension $m$ in $\R^n$
  can be obtained from $\R^m$ by
  a sequence of tropical modifications along   tropical linear spaces of dimension $m-1$.
\end{exa}

\begin{exa}
  The tropical surface in $\R^3$ defined by the tropical polynomial
  $\tg (y+0)z +x+0\td$ is a birational tropical modification of $\R^2$
  along $L_+-L_-$, where $L_+$ (resp. $L-$)
  is the tropical line in $\R^2$ defined
  by the tropical polynomial $\tg x+0\td$ (resp. $\tg y+0\td$), 
  see Figure~\ref{fig:BlowUp}. 
  This surface may be thought as an open part of the blow-up of $\TP^2$
  at the point $(0,0)$, the line $(0,0)\times\R$ corresponding to an
  open part of the exceptional divisor. 
\end{exa}
  
    \begin{figure}[h!]
\begin{center}
\includegraphics[height=7.5cm, angle=0]{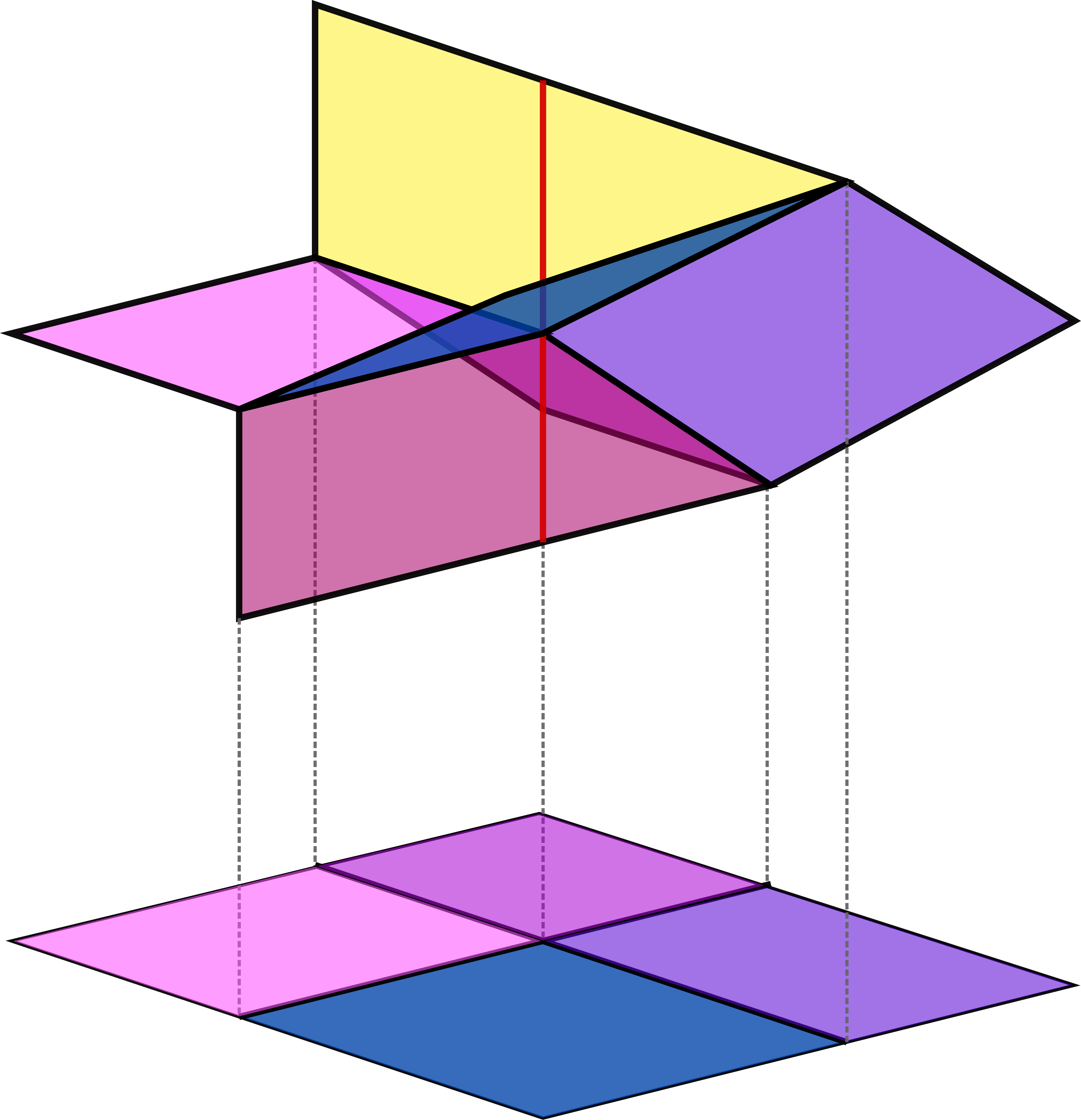}
\end{center}
\caption{The tropical birational modification of $\R^2$ along the divisor $\div_{\tg x+0\td}(\R^2)-\div_{\tg y+0\td}(\R^2)$.  
  The line $(0,0)\times\R$ is contained in the tropical surface.}
\label{fig:BlowUp}
\end{figure}

Given a tropical variety $X$ in $\R^n$ and a divisor $D$ on $X$, 
it is not true in general that there exists a tropical rational
function $f:X\to\R$ such that $\div_X (f)=D$. 
Nevertheless the following proposition
shows that this is true when $X$ is a tropical linear space. This is an immediate generalisation of~\cite[Lemma 2.23]{Sha13-2} 
which treats the case of fan tropical linear spaces. 
The proof from  \cite[Lemma 2.23]{Sha13-2} is based on the following
two facts: 
\begin{itemize}
  \item any fan tropical linear space of dimension $m$ in $\R^n$
   is obtained from $\R^m$ by 
  a sequence of tropical modifications along fan tropical linear
  spaces of dimension $m-1$;
  \item any tropical divisor in $\R^m$ is the divisor of a tropical
    rational function.    
\end{itemize}
Since the first point extends to tropical linear spaces which are
not necessarily fans, the proof of \cite[Lemma 2.23]{Sha13-2} 
extends
immediately as well. 
 
\begin{prop}\label{Thm:function}
  Let $L$ be a tropical linear space in $\RR^\damb$.
  Then any tropical divisor $D$ in $L$ is the divisor of some tropical
  rational function $f:L\to\R$.
\end{prop}
   
   \subsection{Floor composed varieties}

A \emph{construction pattern} 
     is a set $K=\{X_1,\cdots, X_d,D_0, \cdots, D_d, f_1,\cdots, f_d\}$ 
where 
     \begin{itemize}
     \item  $X_i$ is a $\xdim$-dimensional
       connected
       tropical variety in $\R^n$;
     \item $D_{i-1}$ and $D_{i}$ are effective tropical
       divisors on $X_i$, and $f_i:X_i\to \R$ is a tropical rational
       function such that $\div_{X_i}(f_i)=D_i-D_{i-1}$;
     \item $D_i$ is non-empty for $i\in\{1,\cdots,d-1\}$; 
     \item $f_i(p)> f_{i+1}(p)$ for any $p\in D_i$. 
     \end{itemize}

    Note that the above varieties $X_i$ are not disjoint since $D_{i}\subset X_i\cap X_{i+1}$. 
   Given such a construction pattern $K$, we construct a tropical variety 
   $X_K$ of dimension $m$ in $\R^{n+1}$ as follows.   
   For any $i\in \{1,\cdots,d-1\}$, we define $\mathcal W_i$ as the polyhedral complex
 $$\mathcal W_i=\Gamma_{f_i}(D_i)^-\cap \Gamma_{f_{i+1}}(D_i)^+$$
   equipped  with weight inherited from $D_i$. We also define 
$$\mathcal W_d=\Gamma_{f_d}(D_d)^- \qquad\mbox{and}\qquad \mathcal
W_0=\Gamma_{f_1}(D_0)^+$$
equipped with with weight inherited from $D_d$ and  $D_0$ respectively. Finally we define $X_K$ as follows
$$X_K= \mathcal W_0 \cup \bigcup_{i=1}^{d} \left(\Gamma_{f_i}(X_i)\cup
\mathcal W_i\right).   $$
Note that $X_K\subset \displaystyle \bigcup_{i=1}^{d}X_i\times \R$ by construction.
\begin{defi} \label{Defi:FloorComp}
The tropical variety $X_K$ in $\R^{n+1}$ is called 
\emph{the floor composed tropical variety} with pattern $K$.
\end{defi}

\begin{exa} \label{Exa:degree} A classical use of the above
  construction is with a construction pattern $K$ where each $X_i$ is $\R^m$, each divisor $D_i$ is a hypersurface defined 
  by a tropical polynomial $P_i$ of degree  $i$ in $\R^m$, and $f_i=\tg P_i/ P_{i-1}\td$. 
 In this case $X_K$ is a  tropical hypersurface of degree $d$ in $\R^{m+1}$.
An example of such a construction pattern and the corresponding 
floor composed tropical cubic curve in $\R^2$
is depicted in Figure~\ref{fig:FloorComp}.
\end{exa}

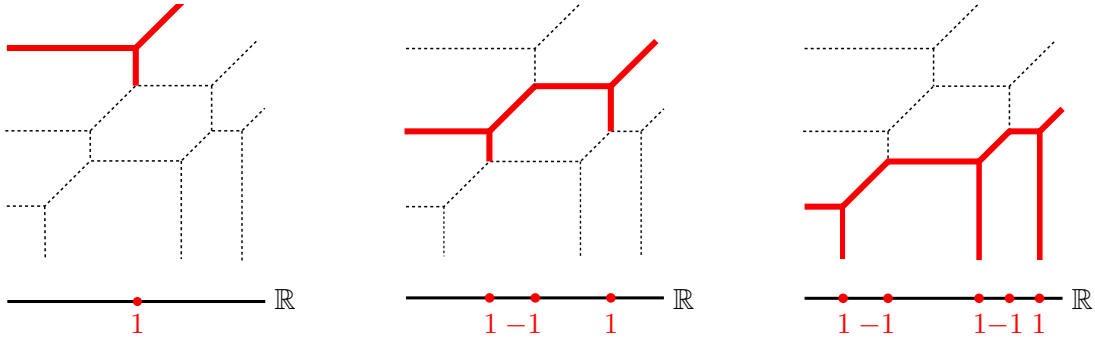
\begin{figure}[h!]
\begin{center}
\begin{tabular}{ccc}
\includegraphics[height=4cm, angle=0]{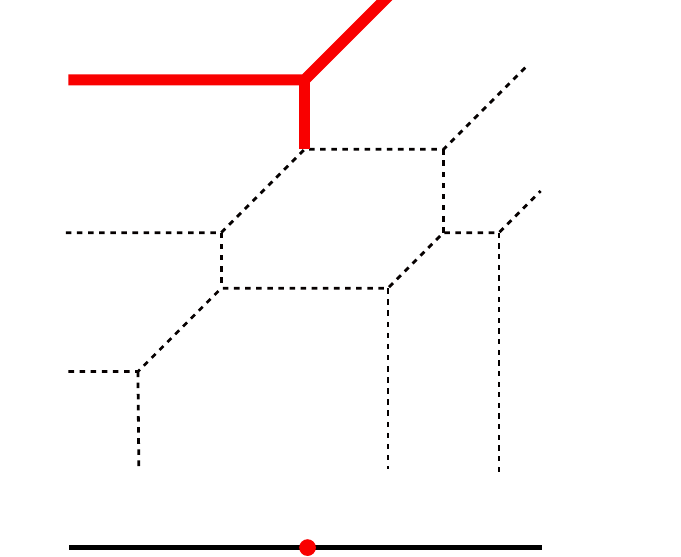}
 \put(-25,-1){$\R$}
 \put(-79,-10){\textcolor{red}{$1$}}
&\includegraphics[height=4cm, angle=0]{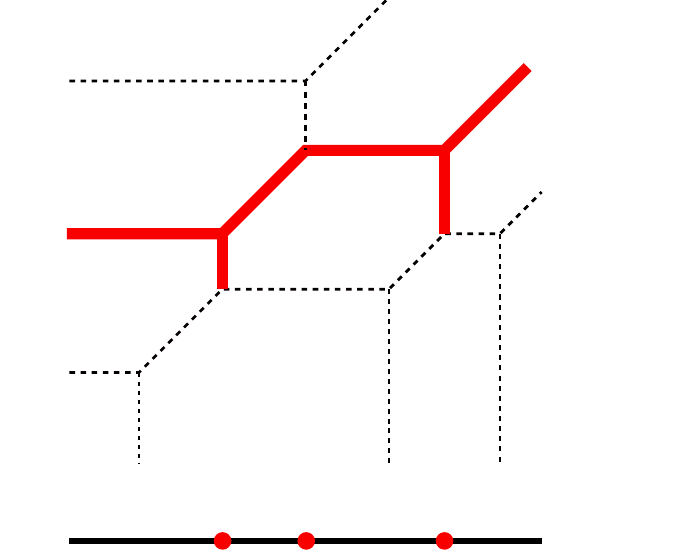}
\put(-25,-1){$\R$}
 \put(-88,-10){\textcolor{red}{$-1$}}
 \put(-96,-10){\textcolor{red}{$1$}}
 \put(-51,-10){\textcolor{red}{$1$}}
&\includegraphics[height=4cm, angle=0]{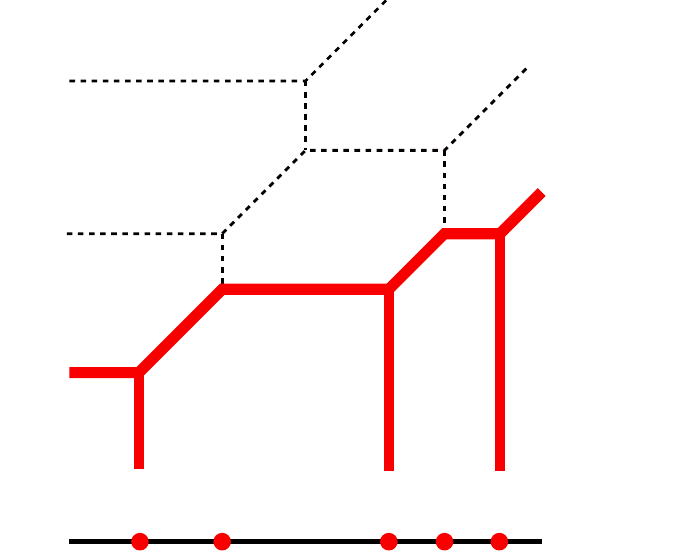}
\put(-25,-1){$\R$}
 \put(-113,-10){\textcolor{red}{$1$}}
 \put(-105,-10){\textcolor{red}{$-1$}}
 \put(-62,-10){\textcolor{red}{$1$}}
\put(-57,-10){\textcolor{red}{$-1$}}
 \put(-40,-10){\textcolor{red}{$1$}}

 \\
\end{tabular}
\end{center}
\caption{A floor composed cubic tropical curve in $\R^2$. 
  For each $X_i=\R$, 
we depicted $D_i-D_{i-1}$ and $\Gamma_{f_i}(\R)\cup\mathcal W_i$} 
\label{fig:FloorComp}
\end{figure}

Our main construction in Section \ref{sec:low any dim} uses a
generalisation of the
previous example with  an arbitrary tropical linear space
in place of $\R^n$.
Given a tropical linear space of dimension $m$ in $\R^n$ and a
surjective linear projection $\pi:L\to\R^m$ to a
coordinate $m$-plane, we denote by $U_\pi\subset \R^m$ the set of
points whose preimage by $\pi$ consists of a single point (it is the
complement in $\R^m$ of an arrangement of at most $n-m$ tropical hyperplanes). 

\begin{defi}\label{Def:function_degree}
Let $L$ be a tropical linear space of dimension $m$ in $\R^n$  and $f:L\to \R$ a tropical rational function. 
The function $f$ is said to have \emph{degree at most $d$} 
if for any surjective linear projection $\pi:L\to\R^m$ to a coordinate $m$-plane, the function $f\circ \pi^{-1}:U_\pi\to\R$ is the 
restriction to $U_\pi$  of a tropical polynomial $P_\pi$ of degree at most $d$ in $\R^m$. It is of \emph{degree  $d$} 
 if it is of degree at most $d$, and not at most $d-1$ (i.e. when at least one of these polynomials has degree $d$).
\end{defi}

Note that with the above definition, not any tropical rational
function $f:L\to\R$  has a degree. This is the case if and only if 
$f$  restricts to a tropical polynomial on every facet of $L$.

\begin{exa}\label{exa:plane}
Let $L$ be the tropical hyperplane in $\R^3$ defined by the tropical
polynomial $\tg x+y+z+0 \td$. The tropical rational function 
$$f(x,y,z)= \begin{array}{c}\tg\\ \ \end{array}\frac{(x+y)\cdot(z+0)}{x+y+z+0}
  \begin{array}{c}\td\\ \ \end{array}  $$
    has degree 1 on $L$. Indeed, by symmetry it is enough to consider
    the projection $\pi(x,y,z)=(x,y)$, and in this case 
    $f\circ \pi^{-1}(x,y)=\tg x+y\td$. Note that $\div_L(f)$ is the
    line $\R(1,1,0)$, see Figure~\ref{fig:Plane}.
\end{exa}

 \begin{figure}[h!]
\begin{center} 
\includegraphics[height=7.5cm, angle=0]{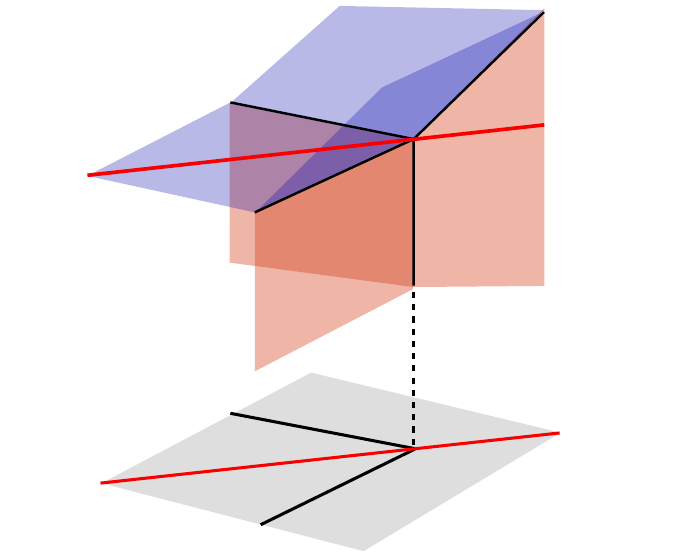}  
\end{center}
\caption{}
\label{fig:Plane}
 \end{figure}

  \begin{lemma}\label{lem:deg div pol}
Let $L$ be a tropical linear space in $\R^n$, and $X$ a tropical
subvariety in $L$ of codimension 1 and degree $d$. Then there exists a
tropical rational 
function $f:L\to \R$ of degree $d$ such that $X=\div_L(f)$.
 \end{lemma}
  \begin{proof}
 Denote by $m$ the dimension of $L$,
and let $\pi:L\to \R^m$ be a surjective linear
projection to a coordinate $m$-plane. For any direction $x_i$
which is
not contracted by $\pi$, we have
\begin{equation}\label{eq:division}
  \begin{array}{c}\tg
  \\\  \end{array}\frac{f}{x_i} \begin{array}{c}\td
  \\\  \end{array}\circ \pi^{-1}=\begin{array}{c}\tg
  \\\  \end{array}\frac{f\circ \pi^{-1}}{x_i} \begin{array}{c}\td
    \\\  \end{array} .
  \end{equation}
Hence we may assume without loss of generality that the tropical
rational function $f$ has a well defined degree but that $\tg f/x_i\td$ 
does not for any index $i$. Suppose now that there exists a projection $\pi_0$ as above 
such that the tropical polynomial  $f\circ \pi_0^{-1}$ has degree  at least  $d+1$.
 Since $\pi_0(X)$ has degree $d$, it follows that there 
exists a direction $x_i$ which is not contracted by $\pi_0$ and such
that $\tg f\circ \pi_0^{-1} /x_i\td$ is still a tropical polynomial. 
But then it follows from $(\ref{eq:division})$ that  
 $\tg f\circ\pi^{-1} /x_i\td$
 is a tropical polynomial for any projection $\pi$
 that does not contract the direction $x_i$. Hence the tropical rational
function $\tg f/x_i\td$ has a well defined degree 
in contradiction with our assumptions. 
 \end{proof}

Next proposition generalises Example \ref{Exa:degree}. 
  \begin{prop}\label{prop:floor deg d}
Let $L$ be a tropical linear space in $\R^n$, 
let $h_0,h_1,\cdots , h_d$ be tropical rational functions on $L$ such that
$h_i$ is of degree $i$, and
let $f_i=\tg h_i/h_{i-1}\td$. 
If $f_i(p)>f_{i+1}(p)$ for any $p\in\div_L(h_i)$, then 
the tropical floor composed variety $X_K$ with pattern 
$K=\{L,\cdots,L,\div_L(h_0) ,\cdots,\div_L(h_d),f_1,$ $\cdots ,f_d \}$
is of degree $d$ in $\R^{n+1}$, and
is  contained in the tropical linear space $L\times \RR$.
\end{prop}
\begin{proof}
  The only thing we have to prove is that $X_K$ is of degree $d$. We
  denote by 
  $m$ the dimension of $L$.
  Let $\Pi$ be a tropical linear space of dimension $n-m$ in $\R^{n}$
  which intersects $L$ in a single point and away from
  $\displaystyle \bigcup_{i=1}^d\mbox{div}_L(h_i)$.
  Hence $\Pi\times \R$ is a tropical linear space in $\R^{n+1}$ which
  intersects $X_K$ in exactly $d$ points, all of them of tropical
  multiplicity 1. 
  The condition that $h_i$ and $h_{i+1}$ have degree
  differing by 1
  ensures that the closures of $X_K$ and $\Pi\times \R$ in
  $\TP^{n+1}$ do not intersect in $\TP^n\setminus \R^n$, and the
  proposition is proved. 
\end{proof}

An $\xdim-$dimensional tropical variety $X$ is called a \emph{homology
bouquet of spheres} if
$$b_0(X)=1\qquad\mbox{and}\qquad b_j(X)=0 \quad \forall j\in
\{1,\cdots,\xdim-1\}.$$
Note that any connected tropical curve is a bouquet of sphere. 
\begin{prop}\label{Thm:Betti}
  Let $K=\{X_1,\cdots, X_d,D_0, \cdots, D_d, f_1,\cdots, f_d\}$ be a
  construction pattern where the tropical varieties  $X_1,\cdots,X_d$
  are homology bouquets of spheres of dimension $m$. Suppose that
  $f_i(p)>f_j(p)$ for any $i<j$ and $p\in X_i\cap X_j$.

If $m=1$, then we have
   $$b_1(X_K)=\sum_{i=1}^db_1(X_i)+\sum_{i=1}^{d-1}\left(b_0(D_i) -1 \right).$$

 If $m\ge 2$ and if  the tropical varieties  $D_0,\cdots,D_d$ are 
  homology bouquets of spheres, then 
 the floor composed tropical variety $X_{K}$ is also a bouquet of
spheres and
$$b_{\xdim}(X_K) = \sum_{i=1}^{d}b_{\xdim}(X_i)+\sum_{i=1}^{d-1}
b_{\xdim-1}(D_i).$$
   \end{prop}
\begin{proof}
  This is an elementary application of the Mayer-Vietoris long exact
  sequence.
We prove the proposition by induction on $d$. The case $d=1$ is clear
since in this case $X_1$ is a deformation retract of $X_K$. 
Let
$K'=\{X_1,\cdots, X_{d-1},D_0, \cdots, D_{d-1}, f_1,\cdots,f_{d-1}\}$ and let us assume that the proposition holds for $X_{K'}$.
  
Defining
$$F_d=\mathcal W_d\cup \Gamma_{f_d}(X_d)\cup\mathcal
   W_{d-1} \qquad\mbox{and}\qquad X^o_{K'}=X_{K'}\setminus
   \Gamma_{f_d}(D_{d-1})^-,$$
   we have
  $$X_K=F_d\cup X^o_{K'} \qquad\mbox{and}\qquad
  \mathcal W_{d-1}=F_d\cap X^o_{K'}. $$
  Figure~\ref{fig:MV} illustrates the above sets on an example.
   Since $D_{d-1}$ (resp. $X_d$, $X^o_{K'}$) is a deformation retract of
   $W_{d-1}$ (resp. $F_d$, $X_{K'}$),
   the Mayer-Vietoris long exact sequence applied to the decomposition $X_{K}=F_d\cup X^o_{K'}$ gives
   \begin{equation}\label{equ:MV betti}
    \cdots \longrightarrow H_{j}(D_{d-1})
   \longrightarrow H_j(X_d) \oplus H_j(X_{K'})
   \longrightarrow H_j(X_{K})\longrightarrow
   H_{j-1}(D_{d-1})\longrightarrow \cdots
   \end{equation}

Since $X_d$ and $X_{K'}$ are
   connected, and  $D_{d-1}$ is non-empty, we deduce that
   the map
   $H_{0}(D_{d-1})\rightarrow H_0(X_d)\oplus H_0(X_{K'})$
   has rank one.
   This proves the result if $m=1$.

   If $m\ge 2$, since 
    $D_{d-1}$ and  $X_d$
   are homology bouquet of spheres, as well as   $X_{K'}$ by induction hypothesis, 
   the long exact sequences $(\ref{equ:MV betti})$ gives
   $$H_0(X_K)\simeq\Z,\qquad H_1(X_K)=\cdots=H_{m-1}(X_K)  =0, $$
   and 
   $$   0 \longrightarrow H_m(X_d) \oplus H_m(X_{K'})\longrightarrow H_m(X_{K}) \longrightarrow H_{m-1}(D_{d-1})\longrightarrow 0.$$
   So the proposition follows by induction on $d$. 
   \end{proof}
   
   \begin{figure}[h!]
\begin{center}
\begin{tabular}{cccc}
\includegraphics[height=2.5cm, angle=0]{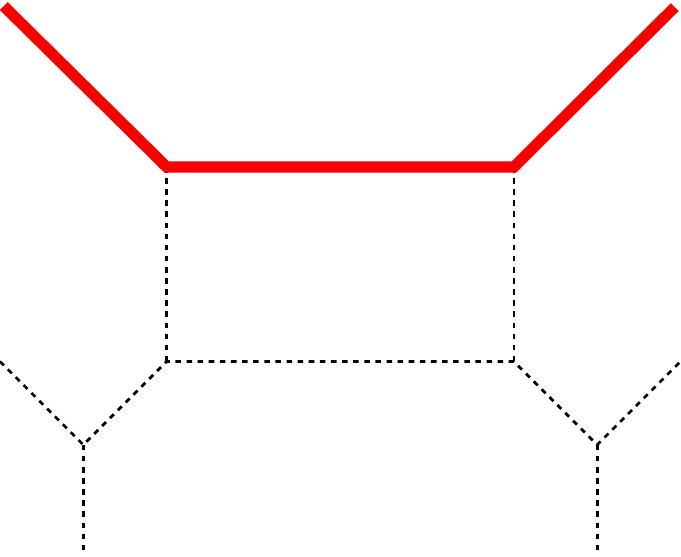}
&\includegraphics[height=2.5cm, angle=0]{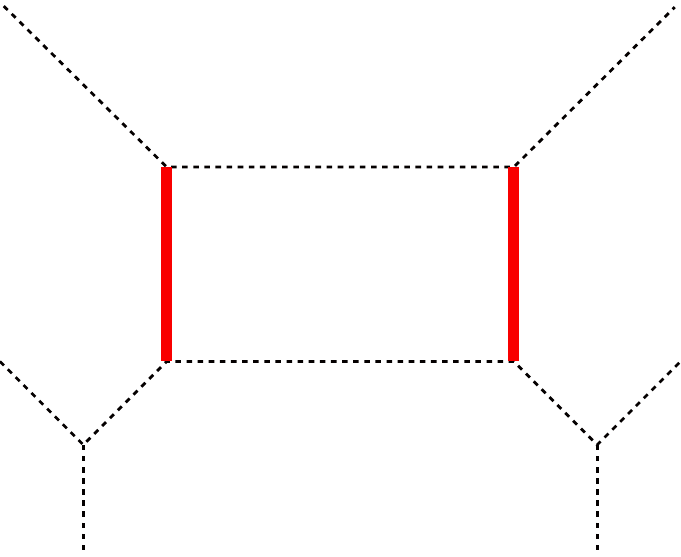}
&\includegraphics[height=2.5cm, angle=0]{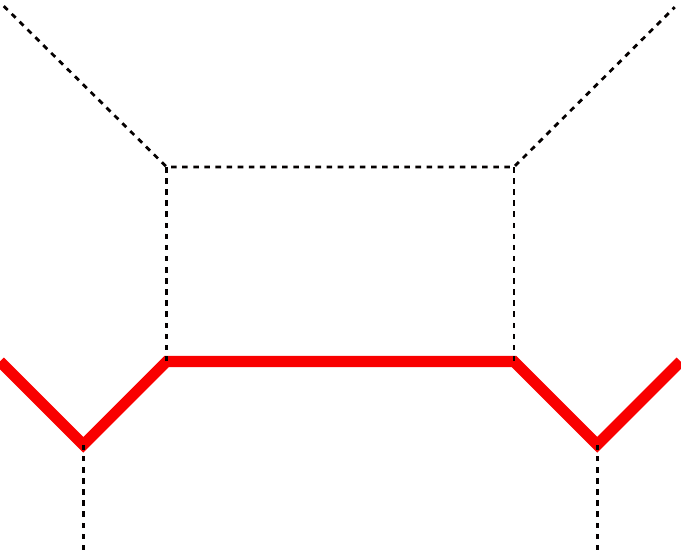}
&\includegraphics[height=2.5cm, angle=0]{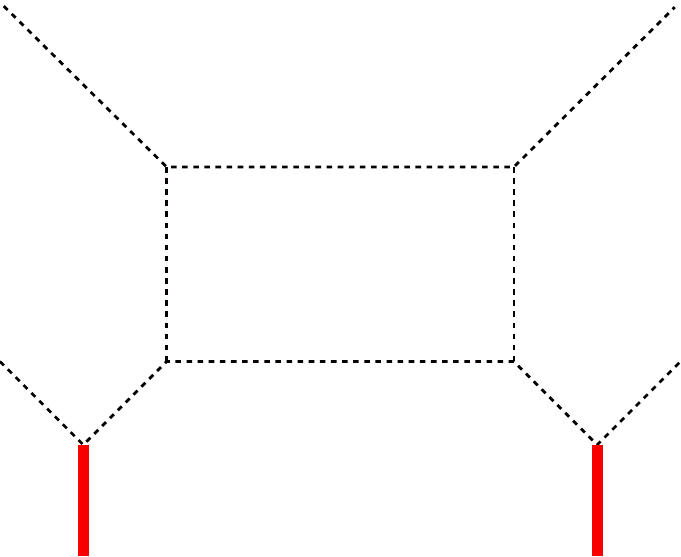}

\\
\\ a) $\Gamma_{f_1}(\R)$ & b) $\mathcal W_1$ & c) $\Gamma_{f_2}(\R)$ & d) $\mathcal W_2$
\\
\includegraphics[height=2.5cm, angle=0]{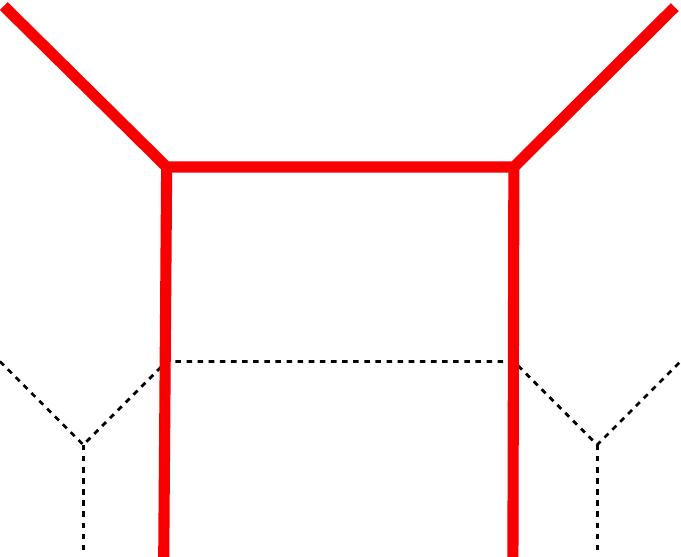}
&\includegraphics[height=2.5cm, angle=0]{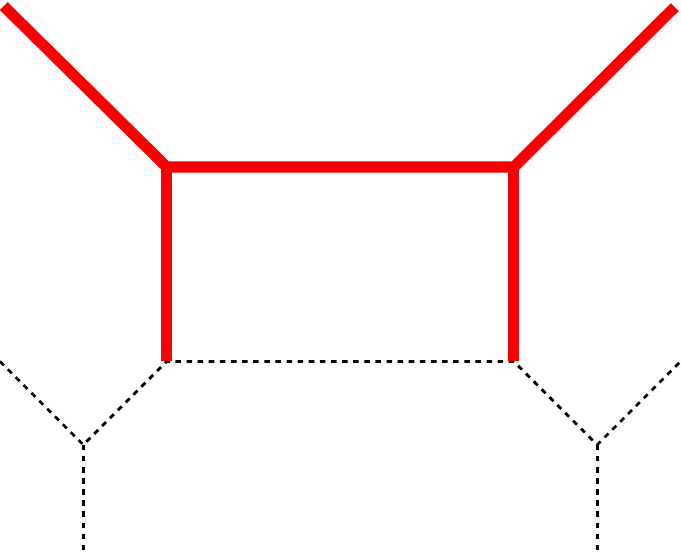}
&\includegraphics[height=2.5cm, angle=0]{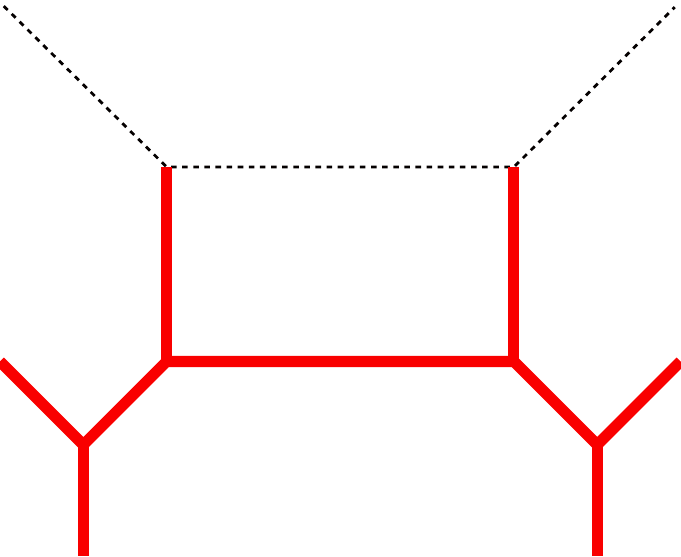}
&\includegraphics[height=2.5cm, angle=0]{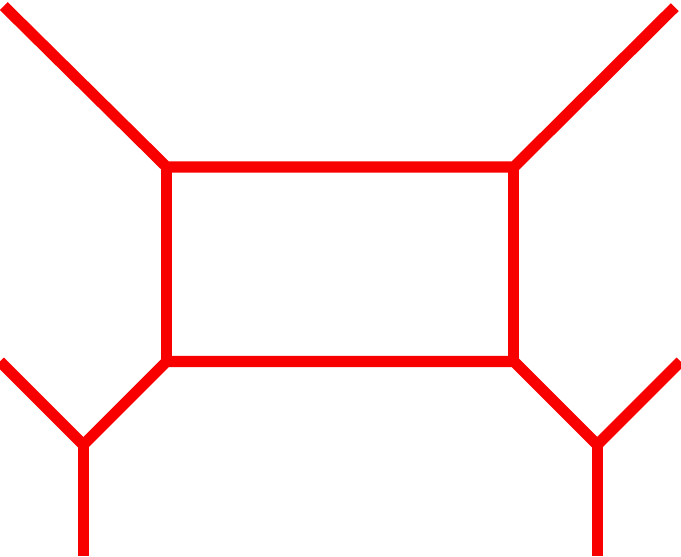}

\\
\\ e) $X_{K'}$ & f) $X_{K'}^o$ & g) $F_2$ & h) $X_K$
\end{tabular}
\end{center}
\caption{Examples of sets defined in the proof of
  Theorem~\ref{Thm:Betti} with $K=\{\R,\R, \emptyset,
  \{-1,1\},\{-2,2\}, \tg1+x+x^{-1}\td,
  \tg\frac{-2x-2x^{-1}}{-1+x+x^{-1}}\td\}$ and $K'=\{\R, \emptyset,
  \{-1,1\}, \tg1+x+x^{-1}\td\}$. 
}
\label{fig:MV}
\end{figure}

\section{Lower estimates}\label{Sec:lower}
The main goal of this section is to prove Theorem
\ref{thm:intro 2}. We first study subvarieties in $\R^n$, the case of curves in Section
\ref{sec:low curves}, from which we deduce
a construction of higher dimensional tropical varieties  by floor
composition in Section \ref{sec:low any dim}. Then we prove Theorem
\ref{thm:intro 2} in Section \ref{sec:proof intro 2}.

Recall that the \emph{recession cone} $R(X)$
of a tropical cycle $X$ in
$\R^n$ is the tropical fan defined by
$$R(X)=\lim_{t\to 0}t \cdot X. $$

\subsection{Curves in $\R^n$}\label{sec:low curves}
Theorem~\ref{thm:curve} 
is contained in
Theorem~\ref{thm:main curve} below. 
In the proof of this latter, we will need the auxiliary families of
curves constructed
in the next two lemmas. 
The conditions regarding intersections in Lemmas
\ref{lem:constr curve 1} and \ref{lem:constr curve 2} 
will be used
in Section~\ref{sec:trop hom} in the proof of Theorem~\ref{thm:main surface}.

The \emph{multiplicity} of a vertex of a tropical curve
in $\R^2$ is  twice the Euclidean area of the  polygon dual to this
vertex. Such a vertex is said to be 
\emph{non-singular} if it has multiplicity 1.
An intersection point $p$ of two tropical curves $C_1$ and $C_2$
in $\R^2$ is said to be
tropically transverse if $p$ is 
a vertex of multiplicity $2$ of $C_1\cup C_2$.
Here we denote by  $L_0$ the tropical line in $\R^2$ 
defined by the tropical polynomial $\tg x+y+0\td$.

\begin{lemma}\label{lem:constr curve 1}
There exists a family of tropical curves $(\widetilde C_d)_{d\ge 1}$
in $\R^2$ satisfying the following properties (see Figure~\ref{fig:constr curve1} for $d=2,3$):
\begin{itemize}
  \item $\widetilde C_1=L_0$;
  \item $\widetilde C_d$ is of degree $d$ and genus $\frac{(d-1)\cdot(d-2)}2$;
  \item $\widetilde C_d$ has an infinite edge $e_\infty$ of weight $d$
    in the direction $(-1,0)$, 
   which is contained in the line  $\{y=0\}$;
  \item   each vertex of $\widetilde C_d$ not adjacent to
  $e_\infty$ is non-singular;
  \item    $\widetilde C_d$ and $\widetilde C_{d-1}$ intersect in exactly $1$ unbounded segment and
    $(d-1)^2$ points, all of them being tropically transverse intersection points;
  \item   $\widetilde C_d$ and $L_0$ intersect in exactly $1$ unbounded segment and $d-1$ points, all of them being tropically transverse intersection points. 
  \item $\displaystyle\bigcap_{d\ge 1} \widetilde C_d$ contains one
    unbounded segments in the direction $(-1,0)$;
  \item $R(\widetilde C_d)=d\cdot L_0$.
\end{itemize}

\begin{figure}[h!]
\begin{center}
\begin{tabular}{ccc}
  \includegraphics[height=4.5cm, angle=0]{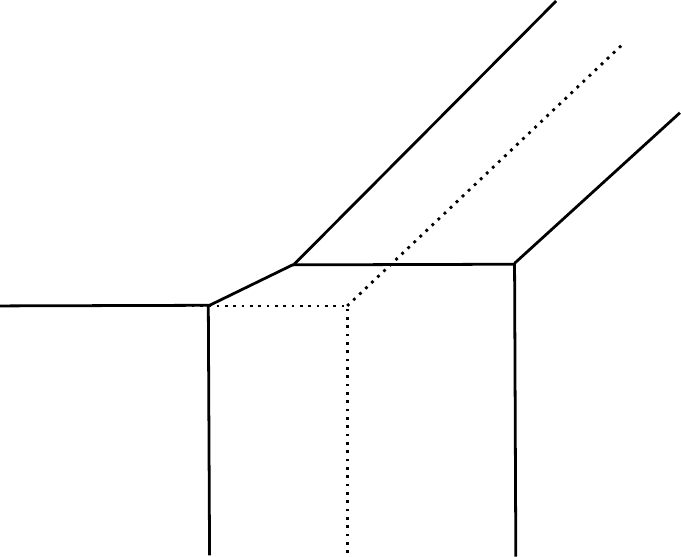}
  \put(-10,120){$L_0$}
   \put(-140,63){$2$}
 \put(-80,100){$\widetilde C_2$}
&\includegraphics[height=4.5cm, angle=0]{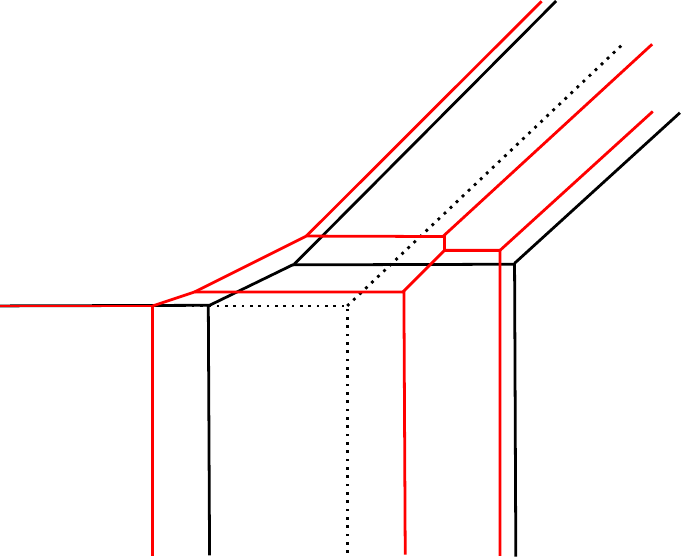}
   \put(-140,63){\textcolor{red}{$3$}}
 \put(-80,100){\textcolor{red}{$\widetilde C_3$}}
&\includegraphics[height=3cm, angle=0]{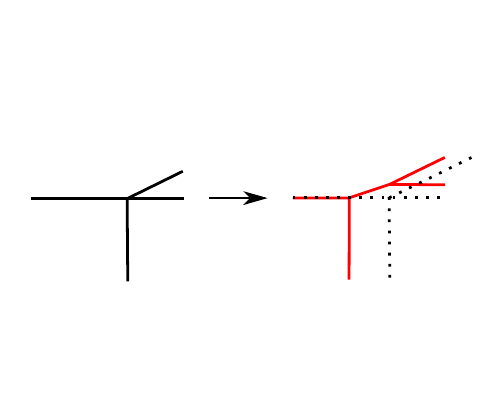}
 \put(-100,8){$\widetilde C_{d-1}\cup L_0$}
    \put(-93,48){$d$}
  \put(-43,48){\textcolor{red}{$d$}}
 \put(-35,8){\textcolor{red}{$\widetilde C_d$}}
\\
\\  a) & b) &c)
\end{tabular}
\end{center}
\caption{} 
\label{fig:constr curve1}
\end{figure}
\end{lemma}

\begin{proof}
The proof is by induction on $d$. For $\widetilde C_2$, we choose the
tropical conic depicted in Figure~\ref{fig:constr curve1}a. To construct the
curve $\widetilde C_d$, we perturb the union of $\widetilde C_{d-1}$
with $L_0$, keeping an edge of multiplicity $d$.
Each non-singular vertex  of $\widetilde C_{d-1}$ gives rise
to a transverse intersection point of $\widetilde C_d$ and
$\widetilde C_{d-1}$. This gives $(d-1)\cdot(d-2)$ such points. Similarly,
each tropically transverse intersection point of $\widetilde C_{d-1}$ and
$L_0$  gives rise to a non-singular vertex of $\widetilde C_d$,  a transverse intersection point of $\widetilde C_d$ and $L_0$, and a transverse intersection point of $\widetilde C_{d}$ and $\widetilde C_{d-1}$. In each case this gives  $d-2$ such intersection points. The vertex of $L_0$ gives rise to a transverse intersection point of $\widetilde C_d$ and $L_0$, hence we have $d-1$ tropically transverse intersection points of $\widetilde C_d$ and $L_0$ as stated.
The vertex adjacent to $e_\infty$ is perturbed as depicted in
Figure~\ref{fig:constr curve1}c, which adds one additional transverse
intersection point of $\widetilde C_d$ and $\widetilde C_{d-1}$.
 The curve $\widetilde C_3$ is depicted on Figure~\ref{fig:constr curve1}b.

 To ensure the last condition, we choose $\widetilde C_d$
 such that the distance between
 the vertex of $L_0$ and every
 vertex of $\widetilde C_d$ is bounded uniformly with respect to $d$. 
\end{proof}

The proof of next lemma is similar to the proof of Lemma~\ref{lem:constr curve 1} and is left to the reader.

\begin{lemma}\label{lem:constr curve 2}
There exists a family of tropical curves $(\overline C_d)_{d\ge 1}$
in $\R^2$ satisfying the following properties (see Figure~\ref{fig:constr curve2} for $d=2,3$):
\begin{itemize}
\item $\overline C_1=L_0$;
\item $\overline C_d$ is of degree $d$ and genus $\frac{(d-1)\cdot(d-2)}2$;
\item $\overline C_d$ has an infinite edge $e_\infty$
  of weight $d$ in the direction $(-1,0)$, 
    which is contained in the line  $\{y=0\}$;
  \item $\overline C_d$ has an infinite edge $e'_\infty$ 
    of weight $d$ in the direction $(1,1)$, 
   which is contained in the line  $\{x=y\}$;
\item each vertex of $\overline C_d$ not adjacent to $e_\infty$
  or $e'_\infty$ is non-singular;
 \item   $\overline  C_d$ and $\overline  C_{d-1}$ intersect in $2$ segments and
  $(d-1)\cdot(d-2)$ points, all them
  being tropically transverse intersection points;
 \item   $\overline  C_d$ and $L_0$ intersect in exactly $2$ segments and $d-2$ points, all of them being tropically transverse intersection points.
 \item $\displaystyle \bigcap_{d\ge 1} \overline C_d$ contains $2$ unbounded segments in directions $-(1,0)$ and $(1,1)$;

\item $R(\overline C_d)=d\cdot L_0$.
\end{itemize}

\begin{figure}[h!]
\begin{center}
\begin{tabular}{cc}
\includegraphics[height=6cm, angle=0]{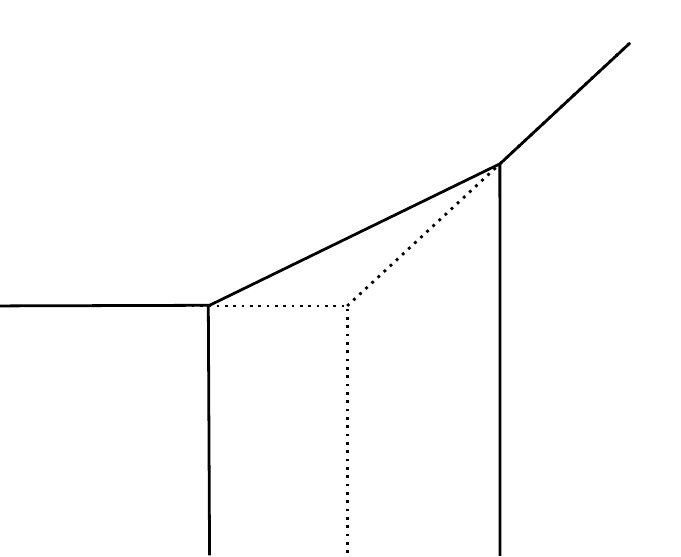}
  \put(-96,20){$L_0$}
   \put(-180,83){$2$}
    \put(-42,143){$2$}
\put(-110,108){$\overline C_2$}
&\includegraphics[height=6cm, angle=0]{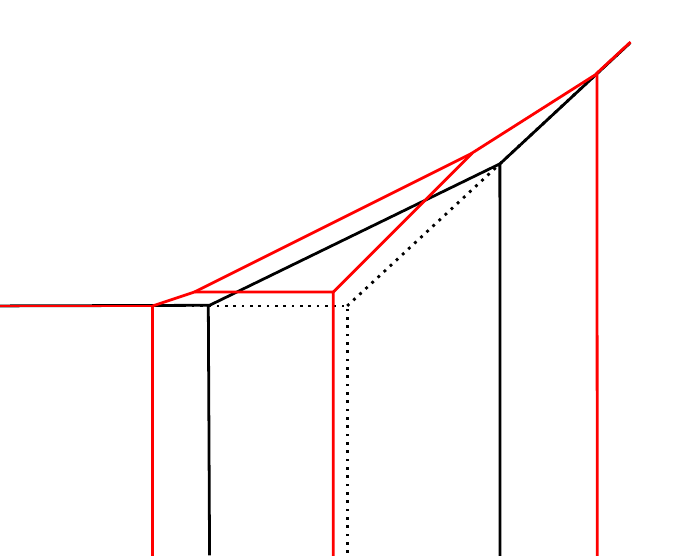}
   \put(-180,83){\textcolor{red}{$3$}}
    \put(-25,158){\textcolor{red}{$3$}}
\put(-120,113){\textcolor{red}{$\overline C_3$}}

 \\
\\  a) & b)
\end{tabular}
\end{center}
\caption{} 
\label{fig:constr curve2}
\end{figure}
\end{lemma}

\begin{thm}\label{thm:main curve}
    For any positive integer $\xcod$,  
  there exists a tropical plane $L_k$ in $\R^{\xcod+1}$ and  a
  family of tropical curves $(C_{k,d})_{d\ge 1}$ in $L_k$ such that 
  (see Figure~\ref{fig:cubicgenus2} for $k=2$ and $d=3$):
  \begin{itemize}
 \item $C_{\xcod,d}$ is tropical curve of degree $d$ and genus
   $\xcod\cdot \frac{(d-1)\cdot (d-2)}2$;
 \item the intersection  $C_{k,d}\cap C_{k,d-1}$ consists of exactly 
   $(d-1)\cdot \big[2(d-1)+ (\xcod-2)\cdot (d-2)\big]$ transverse intersection
   points and $\xcod-1$ segments;
   \item $R(C_{k,d})$ is $d$ times the fan tropical line with one
     unbounded ray in each of the directions
     $$(-1,0,0,\cdots,0), (0,-1,0,\cdots,0), \cdots, (0,\cdots,0,-1), (1,\ldots,1) .$$
    \end{itemize}
\end{thm}

\begin{proof}
  The case $\xcod=1$ is well known, 
   and can be proved for example by perturbing the curves constructed
   in Lemma~\ref{lem:constr curve 1}.
   For $k\ge 2$, we use the following \verb+Gluing+ 
   construction routine.
   We say that a tropical curve $C$ of degree $d$ in $\R^n$ is 
   right-degenerate (resp. left-degenerate)
   if $C$ has an unbounded
edge of weight $d$ in the direction $(1,1,\cdots,1)$ 
(resp. $(-1,0,0,\cdots,0)$) and passing through the origin.
Finally, we denote by  $H_{n-1}$  the tropical hyperplane in
$\R^{n}$  defined by the tropical polynomial
$\tg x_1+\cdots+x_n+0 \td .$

\medskip
\verb+Gluing+ routine

   {\it INPUT }
     \begin{itemize}
     \item a tropical linear plane
       $L$ in $\R^n$;
     \item a right-degenerate tropical curve $C_1$ of degree
         $d$ in $L$;
     \item a left-degenerate tropical curve $C_2$ of degree
         $d$ in $\R^2$.
     \end{itemize}
     
    {\it OUTPUT }  
          \begin{itemize}
          \item a tropical linear plane
            $\widetilde L$ in $\R^{n+1}$;
       \item a  tropical curve $C$ of degree
         $d$ in $\widetilde L$.
     \end{itemize}

         {\it DO}
        \begin{itemize}  
        \item[]
          Let $e_i$ be the edge of $C_i$ passing through the origin.
          Since the multiplicity of intersection at the origin of $H_{n-1}$ (resp $H_1$) and $C_1$ 
 (resp. $C_2$) is $d$, we deduce that $C_1\cap H_{n-1}\subset e_1$ (resp. $C_2\cap H_1\subset e_2$).
We denote by $\widehat C_1$ (resp. $\widehat C_2$) the topological closure of
$C_1\setminus H_{n-1}$ (resp. $C_2\setminus H_1$). 

We embed $\widehat C_1$ and $\widehat C_2$ in $\R^{n+1}$
in such a way that the union of the images is a tropical curve. The
embeddings are given by the two following linear maps: 
$$\gamma_n(x_1,\cdots,x_n)=(x_1,\cdots, x_n,0)\in \R^{n+1} $$
and
$$
\gamma(x,y)=(x,\cdots,x,y)\in \R^{n+1}.$$
We define  $C$ to be the union of the images of
$\widehat C_1$ and $\widehat C_2$ by
$\gamma_\xcod$ and $\gamma$ respectively, equipped with the weights
inherited from  $\widehat C_1$ and $\widehat C_2$. By
construction $C$ is a tropical curve of degree
$d$ contained in $\R^{n+1}$. 
Furthermore, the tropical curve $C$ is contained in
the tropical modification $\widetilde L$ of the tropical plane $L$
along the function  $"x_1+\dots+x_n+0"$.
        \end{itemize}
        
  {\it END}

\medskip
Note that if $C$ is the result of
\verb+Gluing+$(L,C_1,C_2)$, the genus of $C$ is clearly the sum of the  genera of $C_1$ and $C_2$.

 Let  $p=(x_p,0)\in\R^2$  (resp. $q=(x_q,x_q)\in\R^2$)
 be a point that is contained in $e_\infty$ (resp. $e'_\infty)$ of all tropical curves $\widetilde C_d $ 
 and $\overline C_d$ (resp. $\overline C_d$) 
 from Lemmas~\ref{lem:constr curve 1}~and~\ref{lem:constr curve 2}. For  $u \in \R^n$ 
we denote by $\tau_u$ the translation in $\R^n$ by the vector $u$.
Given $d\ge 1$, we define the families of  tropical linear spaces
$(L'_k)_{k\ge 1}$ and of tropical curves $(C'_{k,d} )_{k\ge 1}$ of degree $d$ recursively as
follows: 
\begin{itemize}
\item Let $\widetilde C'_d$ be the tropical curve which is the image of 
$\widetilde C_d$ under the map $r:(x,y)\mapsto (-x,y-x)$, and translated so
  that $r(p)$ is mapped to the origin; 
  set $L'_1=\R^2$ and $C'_{1,d}=\widetilde C'_d$;
  \item $(L'_{k+1},C'_{k+1,d})$ is the translation of
    \verb+Gluing+$(L'_k,C'_{k,d},\tau_{-p}(\overline C_d))$ by the vector
    $(x_p-x_q,\cdots,x_p-x_q,-x_q)$.
\end{itemize}

We define $(L_k,C_{k,d})$ as the output of  \verb+Gluing+$(L'_{k-1},C'_{k-1,d},\tau_{-p}\widetilde C_d)$.

Since the tropical 
curves $\widetilde{C_d}$ and $\overline{C_d}$ are of genus 
$\frac{(d-1)\cdot (d-2)}{2}$, the tropical curve
$C_{k,d}$ is of genus $k\cdot\frac{(d-1)\cdot(d-2)}{2}$.
Each call to \verb+Gluing+ yields one (bounded) segment in
$C_{k,d}\cap C_{k,d-1}$, which thus contains $k-1$ segments. All other
intersections are tropically transverse. 
By Lemmas~\ref{lem:constr curve 1} and \ref{lem:constr curve 2}, the number of tropically
transverse intersection points of $C_{k,d}$ and $C_{k,d-1}$ is equal to
$(k-2)\cdot (d-1)\cdot(d-2) + 2(d-1)^2$. 
By construction, the recession fan $R(C_{k,d})$ is as stated.
\end{proof}

\subsection{Higher dimensional tropical varieties in $\R^n$}\label{sec:low any dim}
We describe in this section an inductive construction of tropical
varieties in $\R^n$ with large Betti numbers, using the curves whose
existence is attested by Theorem \ref{thm:main curve} as the initial
step.
We first need the notion of recession cone of a rational tropical 
function  on a tropical linear space. 
Note that if $L$ is
 a tropical linear space in $\R^n$, then
 there is a canonical one to one correspondence $F \mapsto F^\infty$
between faces of $R(L)$
and unbounded faces of $L$.
 \begin{lemma}
  Let $L$ be a tropical linear space in $\R^n$, and 
  $f:L\to \R$ be a tropical rational function.
Let $u\in R(L)$, and denote by $S(u)$ the union of all faces
of $R(L)$ containing $u$, by $S^\infty(u)$ the union of the
corresponding unbounded faces of $L$, and by
$S^\infty_0(u)$ the set of 
points  $p$ in $S^\infty(u)$ such that
 the half-line
$p+\R_{\ge 0}u$ is contained in 
 $S^\infty(u)\setminus \div_L(f)$. Then the function
  $p\mapsto df_p(u)$ is constant on $S^\infty_0(u)$.
 \end{lemma}
 \begin{proof}
   Let $p_1$ and $p_2$ be two points in  $S^\infty_0(u)$.
   Hence there exists a path from $p_1$
   to $p_2$ in $S^\infty_0(u)$ which crosses $\div_L(f)$ only along its
   facets  containing the direction $u$. By definition of
   $\div_L(f)$, the value of $df_p(u)$ does not change when crossing
   such a facet.
 \end{proof}

  As a consequence, there is a well defined map 
 $$\begin{array}{cccc}
   R(f) : & R(L)&\longrightarrow& \R
   \\& u&\longmapsto & df_p(u)
 \end{array}
 $$
  where
  $p$ is any point in $S^\infty_0(u)$.
  The map $R(f)$ is called the
  \emph{recession map} of $f$.

\begin{thm}\label{Thm:HigherConst}
  For any positive integers
  $\xdim$ and $\xcod$,
  there exist a tropical linear space  $L_{m,k}$
of dimension $m+1$ in $\R^{m+k}$, a tropical linear space $L'_{m,k}$
of dimension $m$  in  $\R^{m+k}$  and  a
  family of tropical hypersurfaces $(X_{m,k,d})_{d\ge 1}$ in $L_{m,k}$
  such that for any $d\ge 1$,
   \begin{itemize}
  \item  $X_{m,k,d}$ is  of degree $d$;
  \item  $X_{m,k,d}$ is a homology bouquet of spheres and
    $$b_\xdim(X_{m,k,d}) =\xcod\cdot B_m(m,1,d) ;$$
  \item $R(X_{m,k,d})=d\cdot L'_{m,k}$.    
  \end{itemize}
\end{thm}
\begin{proof}
  We fix $k$ and we proceed by induction on $m$. The case $m=1$ holds
  by Theorem \ref{thm:main curve}.

  Suppose now that  $L_{m,k}$,  $L'_{m,k}$, and the family $(X_{m,k,d})_{d\ge 0}$
  have been constructed.
  By  Lemma~\ref{lem:deg div pol}, for any $d\ge 0$,
  there exists a tropical rational
  function $h_d:L_{m,k}\to\R$ of degree $d$ such that 
  $\div_{L_{m,k}}(h_d)=X_{m,k,d}$.  
   The recession cone $R(X_{m,k,d}-X_{m,k,d-1})=L'_{m,k}$ does not
  depend on $d$, hence the recession map of $\tg h_d/h_{d-1}\td $ is of degree
  $1$ and does not depend on $d$. In particular, there exists a sequence
  $(\alpha_d)_{d\ge 0}$ of real numbers such that for any sequence
  $(a_d)_{d\ge 0}$ of real numbers satisfying
  $a_{d+1}<a_{d}-\alpha_d$, we have
$$\tg a_{d+1}\cdot h_{d+1}/h_{d} (p)\td <  \tg a_d\cdot h_d/h_{d-1}(p)\td \qquad \forall p\in L_{m,k}.$$ 
  Hence we obtain that the  set
  $$K_{{m,k,d}}=\{L_{k,m},\cdots,L_{k,m}, X_{m,k,0}, X_{m,k,1},\cdots,
  X_{m,k,d},f_1,\cdots,f_d\} $$
  is a construction pattern, where $f_d=\tg a_d\cdot h_d/h_{d-1}\td$
  with $(a_d)_{d\ge 0}$ as above.
  We denote by $X_{m+1,k,d}$ the floor composed tropical variety of
  dimension $m+1$ in
  $\R^{m+k+1}$ with pattern $K_{{m,k,d}}$.
    Recall that $B_m(\xdim,1,d)=\left(\begin{array}{c}d-1\\ m+1\end{array} \right)$. 
    By Proposition \ref{Thm:Betti},  the tropical variety $X_{m+1,k,d}$ is 
  a homology bouquet of spheres and we have
  \begin{eqnarray*}
    b_{\xdim+1}(X_{\xdim+1,k,d})& = &\sum_{i=1}^{d-1} b_{m}(X_{m,k,i})\\
                        & = &\sum_{i=1}^{d-1} \xcod \cdot \left(\begin{array}{c}i-1\\ m+1\end{array} \right)\\
                        & = & \xcod\cdot \left(\begin{array}{c}d-1\\ m+2\end{array} \right)\\
                        & = & \xcod\cdot B_{\xdim+1}(\xdim +1,1,d). 
  \end{eqnarray*}
  Furthermore, by
  Proposition \ref{prop:floor deg d}, $X_{m+1,k,d}$  has degree $d$ and is contained
in the tropical linear space $L_{m+1,k}=L_{m,k}\times \R$.
Since the recession map of $f_d$ is of degree one and does not depend
of $d$, the recession 
fan $R(X_{m+1,k,d})$ is $d$ times a fan tropical linear space
$L'_{m+1,k}$ in $\R^{m+1}$ which does not depend on $d$. Hence the
tropical linear spaces $L_{m+1,k}$ and  $L'_{m+1,k}$, and the family
$(X_{m+1,k,d})_{d\ge 0}$ have been constructed, and the Theorem is proved.
\end{proof}

\subsection{Proof of Theorem \ref{thm:intro 2}}\label{sec:proof intro 2}
 Let $d,m$ and $k$ be three positive integers. We choose $L$ (resp.
 $X$) to be
 the closure in $\TP^n$ of the tropical linear space $L_{m,k}$ (resp. the
 tropical variety $X_{m,k,d}$) from Theorem
 \ref{Thm:HigherConst}.
Since $X\setminus X_{m,k,d}$ is a polyhedral complex of dimension at
most $m-1$, 
we  have
 $$b_m(X)\ge b_m(X_{m,k,d}), $$
 and the Theorem is proved. 
  In the case $m=1$, we furthermore have
 $b_1(X)=b_1(X_{1,k,d})$ since the recession fan $R(X)$ is $d$
 times the fan tropical line with unbounded edges in standard
 directions.
\hfill\BasicTree[1.3]{orange!97!black}{green!90!white}{green!50!white}{leaf}

Theorem~\ref{thm:intro 2} together with Proposition~\ref{prop:finite cells} prove Theorem~\ref{thm:intro} from the introduction.

 \section{Tropical homology of floor composed surfaces}\label{sec:trop hom}
In this section we explicitly compute tropical homology of the
floor composed surfaces constructed in the proof of Theorem~\ref{thm:intro 2}.
We refer to \cite{MikZha14,BIMS15,KSW16} for the definition of  tropical
homology for locally finite polyhedral complexes in the standard
projective space $\TP^n$. 
All tropical homology groups are considered with coefficients in $\R$.
This section partially generalises  results from \cite{Sha13}.

We first start by computing tropical homology of simple tropical
bundles, and apply these results to floor composed surfaces.
Recall that the Mayer-Vietoris Theorem holds for tropical
homology \cite[Proposition 4.2]{Sha13}, and that
an irreducible  compact trivalent\footnote{An irreducible tropical curve $C$ in $\TP^n$
  is said to be trivalent
  if $\val_p(C)\le 3$ for every point $p\in C$.} tropical curve of genus $g$ has the following
tropical Hodge diamond:
$$\begin{array}{ccc}
   & 1 & 
\\   g & &  g
\\   &1 & 
\end{array} $$

\subsection{Tropical homology of tropical ruled varieties}\label{sec:bundle}

We denote by $\Delta_n$ the standard  unimodular simplex in $\R^n$, and by 
$\widetilde \Delta_{n}^i$  the convex  polytope in $\R^{n}$ which is the
convex hull of the union of $i\cdot \Delta_{n-1}\times\{0\}$ and
$\Delta_{n-1}\times\{1\}$. The corresponding algebraic toric variety
is
$$\mathbb P\left(\mathcal O_{\CP^{n-1}}(i)\oplus \mathcal
O_{\CP^{n-1}} \right)=
\mathbb P\left(\mathcal O_{\CP^{n-1}}(-i)\oplus \mathcal
O_{\CP^{n-1}} \right).$$
We denote by $\mathbb T \widetilde \Delta_{n}^i$ the corresponding 
tropical toric variety. The faces $i\cdot \Delta_{n-1}\times\{0\}$ and
$\Delta_{n-1}\times\{1\}$ of $\widetilde \Delta_{n}^i$ correspond to
two divisors of $\mathbb T \widetilde \Delta_{n}^i$, respectively
 denoted by $E_-$ and $E_+$, that are contained in the
 boundary of $\mathbb T \widetilde \Delta_{n}^i$. Note that both 
 $E_-$ and $E_+$ are equal to  $\TP^{n-1}$. Furthermore, 
  there are two  natural projections
 $\pi_\pm:\mathbb T \widetilde \Delta_{n}^i\to E_\pm$, 
 which are tropical morphisms, and whose fibre over any point is  $\TP^1$.
 
\begin{exa}
  The standard tropical Hirzebruch surface $\mathbb T\mathbb F_i$ of degree $i$ is
  defined as $\mathbb T \widetilde \Delta_{2}^i$.
  Note that the divisor $E_+$ is tropically linearly equivalent
  (see for example \cite[Section 4.3]{Mik06} or \cite[Section 6.3]{MikRau19})
  to the divisor
$E_-+iF$, where $F$ is any fibre of $\pi_\pm$.
\end{exa}

\begin{defi}
  Let $X$ be a tropical variety in  $\TP^n$ identified with 
  $E_-\subset  \mathbb T\widetilde \Delta_{n}^i$.
\begin{itemize}
\item  The cylinder $\Sigma=\pi_{-}^{-1}(X)$  over $X$ in
$\mathbb T \widetilde \Delta_{n}^i$ is called a \emph{$\TP^1$-bundle
    over $X$}.
 The intersection of  $\Sigma$ with $E_{\pm}$  is denoted by $X_{\pm}$.

\item The tropical varieties $\Sigma_{-}=\Sigma\setminus X_{+}$ and
 $\Sigma_{+}=\Sigma\setminus X_{-}$ are called \emph{tropical line bundles
  over $X$}.
  
\item The tropical variety
 $\Sigma^{oo}=\Sigma_-\cap \Sigma_+$
  is called a \emph{$\mathbb T^\times$-bundle over $X$}.
  \end{itemize}
\end{defi}
This is a rather restrictive notion of $\TP^1$/line/$\mathbb
T^\times$ bundles, however it will be sufficient for our purposes. We
refer for example to \cite{Mik6,All09} for a more general definition of
tropical line bundles.

A $\TP^1$-bundle $\Sigma$ over a projective tropical variety $X$ comes naturally equipped with two natural
 tropical projections $\pi_\pm:\Sigma \to X_{\pm}$ with a section 
 $\iota_{\pm}:X\to X_{\pm}\subset \Sigma$. 

We compute, in the following lemmas, tropical homology groups of $\TP^1$, line and $\mathbb T^\times$ bundles. 

\begin{lemma}\label{lem:line bd}
  Let $\Sigma_\pm$ be a tropical line bundle over a tropical variety $X$.
Then for any pair $(p,q)$,   the inclusion
  $\iota_{ \pm}$ induces an isomorphism 
   $$\iota_{\pm *}:H_{p,q}(X)\simeq H_{p,q}(\Sigma_\pm).$$
\end{lemma}
\begin{proof}
The morphism $\iota_{\pm *}$ is injective since it is clearly a section  
  of the morphism $H_{p,q}(\Sigma_\pm)\to H_{p,q}(X)$ induced by 
the projection $\pi_\pm$.

Equip $\Sigma_\pm$ with any locally finite
polyhedral subdivision compatible with its tropical structure. Recall that the
cellular tropical homology of $\Sigma_\pm$  is isomorphic to the singular
tropical homology of $\Sigma_\pm$, and  is thus independent of the chosen 
 subdivision \cite[Proposition 2.2]{MikZha14}.
A $(p,q)$-cell $\sigma$ in $\Sigma_\pm$ is called vertical
if $\pi_\pm(\sigma)$ has dimension strictly less than $q$.
A $(p,q)$-chain  in $\Sigma_\pm$ is called  vertical
 if  every cell in its support is vertical. Any $(p,q)$-chain in $\Sigma_\pm$ is homologous to the sum of a
$(p,q)$-chain with support in  $X_\pm$ and a   vertical $(p,q)$-chain. Since no vertical chain in
$\Sigma_\pm$ can be closed, we obtain that  any $(p,q)$-cycle in $\Sigma_\pm$ can be represented
  by a $(p,q)$-cycle in $X_\pm$. In other words, the map $\iota_{\pm *}$ is
  surjective and is thus an isomorphism. 
\end{proof}

Let $\Sigma$ be a $\TP^1$-bundle over a tropical variety $X$, and let $u_-$ 
be the  primitive integer vector generating the kernel of 
$d\pi_-$ and pointing away from $X_-$ (there is a unique 
choice of such a vector in each tropical tangent space of $\Sigma$).
To a  $(p-1,q-1)$-cell $\sigma=\beta_Q\cdot Q$ in $X$, with $Q$ a
$(q-1)$-dimensional face of $X$ and $\beta_Q \in \mathcal F_{p-1}(Q)$, 
 we associate the $(p,q)$-cell $\kappa(\sigma)=(u_-\wedge \beta_Q)\cdot \pi_-^{-1}(\iota_-(Q))$
 in $\Sigma$, where the orientation of $\pi_-^{-1}(\iota_-(Q))$ is 
 induced by the orientation on $\iota_-(Q)$. 
 This induces a linear map
$$\kappa : H_{p-1,q-1}(X)\to H_{p,q}(\Sigma),$$
 that we call a tropical Gysin map. Note that the
tropical Gysin map is the same if one defines it
 using the section $\iota_+$ instead of $\iota_-$. Furthermore it maps
 straight classes (i.e. classes induced by tropical cycles) of $X$ to straight classes of $\Sigma$.
 The inclusion map $\iota_{\pm}:X\to \Sigma$ induces a linear map
 $H_{p,q}(X)\to H_{p,q}(\Sigma)$ that we still denote by
 $\iota_{\pm *}$ to avoid additional notations.
 This slight abuse of notation is justified in particular by next lemma.
 
\begin{lemma}\label{lem:proj bd}
 For any $\TP^1$-bundle $\Sigma$ over a tropical variety $X$, and for any pair $(p,q)$, the maps $\iota_{- *}$ and $\kappa$ induce an isomorphism
 $$ (\iota_{- *},\kappa):\,H_{p,q}(X)\times H_{p-1,q-1}(X)\simeq H_{p,q}(\Sigma).$$
\end{lemma}
\begin{proof}
 The map $\iota_{- *}$ is injective since it is a section of $\pi_{- *}$.
As $X$ and $\Sigma$ are both compact, we choose their polyhedral
subdivision induced by the tropical structure on $X$.
 As in the proof of Lemma \ref{lem:line bd}, any $(p,q)$-chain
$\sigma$ 
in $\Sigma$
is homologous to the sum of a $(p,q)$-chain $\sigma_-$ 
in $X_-$ and a vertical $(p,q)$-chain $\sigma_v$.

Suppose that $\sigma$ is a $(p,q)$-cycle in $\Sigma$. The cellular
boundary of any vertical cell of $\Sigma$ intersects $X_+$ which is disjoint from
$X_-$. Hence the vector $u_-$ divides the framing of each cell
contained in the support of $\sigma_v$, that is to say
$\sigma_v=\kappa(\sigma_0)$ with  $\sigma_0$ a $(p-1,q-1)$-chain in
$X$.
In turn, this implies that the support of $\partial \sigma_v$ is
disjoint from $X_-$, from which we deduce that
$$\partial \sigma_-=\partial \sigma_v=0. $$
This proves that the map
$\iota_{- *}\times \kappa$ is surjective.

Conversely, suppose that $\sigma'$ and $\sigma''$ are respectively 
$(p,q)$ and $(p-1,q-1)$-cycles in $X$ such that
$$\iota_{- *}(\sigma') + \kappa(\sigma'')=\partial \gamma. $$
As above, we have 
 $\iota_{- *}(\sigma')=\partial \gamma_-$ and
$\kappa(\sigma'')=\partial \gamma_v=\kappa(\partial \gamma_0)$, which 
further implies that both  $\sigma'$ and $\sigma''$ are null homologous.
Hence the map
$\iota_{- *}\times \kappa$ is injective, and the lemma is proved.
\end{proof}

The map $\kappa$ does not depend on which section $\iota_-$ or
$\iota_+$ we choose to define it, however the inclusion
$H_{p,q}(X)\to H_{p,q}(\Sigma)$ does.
Let 
$$\nu_{p,q}: H_{p,q}(X)\to H_{p-1,q-1}(X)$$
be the linear map obtained by the following compositions 
$$H_{p,q}(X)\xrightarrow[\phantom{blablabla}]{\iota_{+ *}}
H_{p,q}(\Sigma)\xrightarrow[\phantom{blablabla}]{{(\iota_{- *},\kappa)}^{-1}} H_{p,q}(X)\times
H_{p-1,q-1}(X) \xrightarrow[\phantom{blablabla}]{} H_{p-1,q-1}(X),$$
where the last map is the projection on the second factor.
Note that $\nu_{p,q}$ is the zero map if and only if
$\iota_{+ *}=\iota_{- *}$. The image of $\nu_{\dim X,\dim X}$ is called
\emph{the first Chern class} of the tropical line bundle $\Sigma_-$
(and so it is minus the first Chern class of the line bundle $\Sigma_+$).

\begin{exa}\label{ex:chern surface}
  Consider the tropical Hirzebruch surface $\mathbb T\mathbb F_i$
  of degree $i$.
Recall that the divisor $E_+$ is tropically linearly equivalent to the divisor
$E_-+iF$, where $F$ is the divisor of $\mathbb T\mathbb F_i$ corresponding to
the side $[(0,0);(0,1)]$ of $\widetilde \Delta_2^i$. Hence the
corresponding straight classes satisfy 
$$[E_+]= [E_-] +i[F] $$
in $H_{1,1}(\mathbb T\mathbb F_i)$. In particular, the first Chern class of
$\mathbb T\mathbb F_i\setminus E_+$ is $i$ times the class of a point.

More generally, let $\Sigma\subset \widetilde \Delta_n^i$ be a $\TP^1$-bundle over a
compact tropical curve of degree $d$
in $E_-$. It follows from the balancing condition that the first
 Chern class of $\Sigma_-$ is equal to $i\cdot d$ times the class of a point.
\end{exa}

Next we turn to tropical homology of $\mathbb T^\times$-bundles.
\begin{cor}\label{cor:torus bd}
 For any $\TP^1$-bundle $\Sigma$ over a tropical variety $X$, and for any pair 
 $(p,q)$, one has the isomorphism
  $$H_{p,q}(\Sigma^{oo})\simeq \mbox{Ker }\nu_{p,q} \times
  \left(H_{p-1,q}(X)/\mbox{Im }\nu_{p,q+1} \right).$$
\end{cor}
\begin{proof}
   The Mayer-Vietoris Theorem applied to the triple $(\Sigma,\Sigma_-,\Sigma_+)$
      gives the long exact sequence
      \begin{equation}\label{eq:MV1}
        \ldots\longrightarrow H_{p,q}(\Sigma^{oo})\longrightarrow
H_{p,q}(\Sigma_-)\times H_{p,q}(\Sigma_+) \longrightarrow
H_{p,q}(\Sigma)\longrightarrow
H_{p,q-1}(\Sigma^{oo})\longrightarrow\ldots
\end{equation}
      By Lemma \ref{lem:line bd}, we have canonical isomorphisms
      $\iota_{\pm *}:H_{p,q}(X)\to H_{p,q}(\Sigma_\pm)$.
      By Lemma \ref{lem:proj bd}, we have an isomorphism 
  $\iota_{- *}\times\kappa  :H_{p,q}(X)\times H_{p-1,q-1}(X)\to
      H_{p,q}(\Sigma)$. With these identifications,  the
image of the map
$$H_{p,q}(\Sigma_-)\times H_{p,q}(\Sigma_+) \longrightarrow
H_{p,q}(\Sigma)$$
is precisely  $H_{p,q}(X)\times \mbox{Im }\nu_{p,q}$. Hence the long exact sequence
$(\ref{eq:MV1})$ splits into the short exact sequences
$$0\longrightarrow  H_{p-1,q}(X)/\mbox{Im }\nu_{p,q+1} \longrightarrow
H_{p,q}(\Sigma^{oo})\longrightarrow   \mbox{Ker }\nu_{p,q} \longrightarrow 0,$$
and the result follows.
\end{proof}

\begin{exa}
  In the extremal  cases  when $p=\dim X+1$, or $p=0$, or $q=\dim X+1$, Corollary
  \ref{cor:torus bd} gives
$$H_{\dim \Sigma,q}(\Sigma^{oo})=H_{\dim X,q}(X), \quad
  H_{p,\dim \Sigma}(\Sigma^{oo})=0,\quad
  \mbox{and}\quad H_{0,q}(\Sigma^{oo})=H_{0,q}(X).$$
\end{exa}

\begin{exa}\label{ex:torus bd curve}
  Suppose that $X$ is a compact trivalent tropical curve of genus $g$. Then
  Corollary~\ref{cor:torus bd} gives the
  following tropical Hodge diamond for $\Sigma^{oo}$ 
  (as in the introduction, by convention, $h_{0,0}$ is the topmost number and $h_{2,0}$ the lefmost):
$$\begin{array}{ccccc}
  & & 1 & &
\\  & g+\varepsilon & & \phantom{+} g\phantom{\varepsilon} &
\\g  & & g +\varepsilon & & 0
\\  & 1 & & 0 &
\\  & &0 & &
\end{array} $$
where $\varepsilon=0$ if  the  
first Chern class of $\Sigma_-$ does not vanish, and  $\varepsilon=1$ if
it does. Note that this example corrects a small mistake in \cite[Lemma
  4.3ii)]{Sha13}.
\end{exa}

\subsection{Tropical homology of birational tropical modifications} \label{sec:hom bir modif}
The method we used in Section~\ref{sec:bundle} also  allows the 
 computation of tropical homology of a birational tropical
 modification of a  tropical variety.
 Recall that $ \pi_-:\mathbb T\widetilde \Delta_{n}^i\to E_-$ is a
 $\TP^1$-bundle over $E_-=\TP^{n-1}$.
As in Section \ref{sec:bundle}, we denote 
 by $u_-$ the  primitive integer vector generating the kernel of
$d\pi_-$ and pointing away from $E_-$. 
If $Y$ is  a tropical variety in $\mathbb T\widetilde \Delta_{n}^i$,
we denote by $Y_\pm$ its intersection with the divisor $E_\pm$,  and by
$Y_+^p$ the tropical variety $\pi_-(Y_+)$.
\begin{defi}
A  tropical variety $Y$ in $\mathbb T\widetilde \Delta_{n}^i$ is
called a \emph{birational tropical
modification} of $X\subset E_-$ along the divisor $Y_--Y_+^p$ if
$Y\cap \R^n$
 is a birational tropical
 modification of $X\cap \R^{n-1}$ along the divisor
 $(Y_--Y_+^p)\cap \R^{n-1}$, and if $Y$ is the topological closure of $Y\cap\R^n$ in
 $\mathbb T\widetilde \Delta_{n}^i$.

 If $Y_+=\emptyset$, then $Y$ is called a tropical modification of $X$
 along the divisor $Y_-$.
\end{defi}
Given such a birational tropical
modification $Y$ of $X$, we still denote by $\pi_-$ the restriction
of $\pi_-$ to $Y$. 
We emphasise 
that in the next proposition, it is not assumed that the tropical
prevariety $Y_-\cap Y_+^p$ is a tropical variety
(recall that a
 tropical variety is defined as the set-theoretic intersection of some
 tropical varieties, see \cite[Section 3]{RGJS05}).
\begin{lemma}\label{lem:hom bir modif}
  Let $Y\subset \mathbb T\widetilde \Delta_{n}^i$ be a birational tropical
  modification of $X\subset E_-$  along the divisor $Y_--Y_+^p$.
  Then for any pair $(p,q)$, one has
  $$H_{p,q}(Y)\simeq H_{p,q}(X) \times H_{p-1,q-1}(Y_-\cap Y_+^p). $$
\end{lemma}
\begin{proof}
  Since all tropical varieties involved are compact,
  we choose their polyhedral
subdivision induced by their tropical structure.
  The map $\pi_-:Y\to X$ induces a map on the  chain groups
  $$\pi_{- *} : C_q(Y, \F_p) \to   C_q(X, \F_p)$$
  that commutes with the boundary map.
  We denote by
 $\widetilde Y$ the union of all
faces of $Y$ on which  $d\pi_-$ is
injective, i.e. $\widetilde Y$ is the union of faces of $Y$ 
 on which the restriction of $\pi_-$ is a bijection.
We denote by $\tau$ the inverse map of $\pi_{-|\widetilde Y}$.

\medskip
  We start by constructing a  section $s$ of the map
  $\pi_{- *} : H_{p,q}(Y) \to   H_{p,q}(X)$.
  Given a $(p,q)$-cell $\sigma$  in $X$, choose a facet $F_\sigma$ of
  $X$ containing the support of $\sigma$. Then $\sigma$
  induces via $\tau|_{F_\sigma}$ a
  $(p,q)$-cell $\tau_*(\sigma)$  in $Y$.
  Note that a different choice (if any) of $F_\sigma$ gives rise to a different
  $(p,q)$-chain,  differing from $\tau_*(\sigma)$ by a framing divisible
  by $u_-$; this will not be important in what follows.
  By linearity,
  we
  obtain a linear map
    $$\tau_{ *} : C_q(X, \F_p) \to   C_q(Y, \F_p).$$
  If $\sigma$ is a $(p,q)$-cycle in $X$, then by construction
$\partial \tau_*(\sigma)$ 
  has support contained in
$\pi_-^{-1}(Y_-\cup Y_+^p)$
  and has a framing divisible by $u_-$.
Hence $\partial \tau_*(\sigma)$ is the boundary in $Y$ of a vertical $(p,q)$-chain $ \sigma_v$, and
  we define
  $$ s(\sigma)=\tau_*(\sigma) - \sigma_v. $$
  The map $s$ is a section of the map $\pi_*$, in particular it is
  injective. In the rest of the proof we identify $H_{p,q}(X)$ and its
  image by $s$ in $H_{p,q}(Y)$.

  \medskip
  Next,  the same construction than the
  construction of the tropical Gysin map in Section \ref{sec:bundle}
  provides a linear map
  $$ \kappa : H_{p-1,q-1}(Y_- \cap Y_+^p) \to   H_{p,q}(Y).$$
  With a proof analogous to the proof in
  Lemma \ref{lem:proj bd} that the map $\iota_+\times \kappa$ is an
  isomorphism, we obtain that
   the linear map
  $s\times \kappa: H_{p,q}(X) \times H_{p-1,q-1}(Y_-\cap Y_+^p)\to
  H_{p,q}(Y)$ is also an isomorphism. 
\end{proof}

Applying Lemma \ref{lem:hom bir modif} in the particular case when
$Y_+$ is empty, we recover the result by Shaw that  tropical homology
groups are invariant 
under tropical modifications.
\begin{cor}[Shaw, {\cite[Theorem 4.13]{Sha15}}]
  Let $Y\subset \mathbb T\widetilde \Delta_{n+1}^i$ be a tropical
  modification of $X\subset E_-$.
  Then for any pair $(p,q)$, the linear map
  $$\pi_{- *}:H_{p,q}(Y)\longrightarrow H_{p,q}(X) $$
  is an isomorphism.
\end{cor}

Any tropical linear space of dimension $m$ in $\TP^n$
is obtained from  $\TP^m$ by a finite sequence of
tropical modifications along linear tropical divisors, hence they have
the same tropical Hodge diamond. There are many ways to compute tropical
homology groups of $\TP^m$ (see for example
\cite[Example 7.27]{BIMS15} and \cite[Corollary 2]{IKMZ}),  with which 
we obtain the following well-known statement. 
\begin{cor}\label{cor:hom linear}
  Let $L$ be a tropical linear space of dimension $m$ in 
  $\TP^n$. Then one has
  $$h_{p,p}(L)=1 \ \ \forall  p=0,1, \ldots, n,
  \qquad\mbox{and}\qquad h_{p,q}(L)=0 \mbox{  otherwise}. $$
\end{cor}

\subsection{Back to tropical surfaces}

Now we specialise results from Sections \ref{sec:bundle} and
\ref{sec:hom bir modif} to the case of floor
composed tropical surfaces. 
Throughout the whole section, we
consider the family of tropical curves $(C_{k,d})_{d\ge 1}$ in
$\TP^{k+1}$ we constructed in
Theorem~\ref{thm:main curve}, and the tropical plane $L_k$ which
contains them. 
We denote by $\Sigma_{k,d}$ the $\TP^1$-bundle over $C_{k,d}$ in 
$\mathbb T\widetilde \Delta_{k+2}^1$, and by $L_{k,d,d-1}$ the
birational tropical modification of $L_k$ along $C_{k,d} -
C_{k,d-1}$. 
\begin{lemma}\label{lem:remove boundary modif}
For any integer $k\ge 1$ and $d\ge 2$,  the tropical Hodge diamond of
  $L_{k,d,d-1}^{o}= L_{k,d,d-1}\setminus E_+$  is the following 
  $$\begin{array}{ccccc}
  & & 1 & &
\\  & 0 & & 0 &
\\ k\cdot g(C_{1,d-1}) & & k\cdot\big[d\cdot(d-1)+
g(C_{1,d-1})\big]+(k-1)\cdot (2d-3)  &\phantom{h_{2,1}(X)} & 0
\\  & k-1 & & 0 &
\\  & &0 & &
  \end{array} $$ 
Furthermore, both natural maps
$H_{2,0}(\Sigma_{k,d-1}^{oo})\to H_{2,0}(L_{k,d,d-1}^{o})$ and
$H_{1,1}(\Sigma_{k,d-1}^{oo})\to H_{1,1}(L_{k,d,d-1}^{o})$
are  injective.
\end{lemma}
\begin{proof}
The case $p=0$ is clear since
$H_{0,q}(L_{k,d,d-1}^{o})=H_q(L_{k,d,d-1}^{o};\R)$ and that $L_{k,d,d-1}^{o}$
is contractible.
The non-vanishing tropical Hodge numbers of a segment in $\R^n$ are
precisely $h_{0,0}=h_{1,0}=1$.
Hence $L_{k,d,d-1}$ has the following tropical Hodge diamond by Lemma
\ref{lem:hom bir modif} and Theorem \ref{thm:main curve}
  $$\begin{array}{ccccc}
  & & 1 & &
\\  & 0 & & 0 &
\\ 0 & &k\cdot\big[1 + d\cdot(d-1)\big] -2(k-1)\cdot(d-1)  &\phantom{k-1} & 0
\\  & k-1 & & 0 &
\\  & &1 & &
  \end{array} $$ 
The first Chern class of $\Sigma_{k,d}$ is non-null
by Example \ref{ex:chern surface}.
We consider the decomposition of $L_{k,d,d-1}$ into the union of
$L_{k,d,d-1}^o$ and of a connected and simply connected neighbourhood of
$C_{k,d-1}$ in $L_{k,d,d-1}$.
The Mayer-Vietoris sequence together with Lemma
\ref{lem:line bd} and Example \ref{ex:torus bd curve} give that
$H_{1,2}(L_{k,d,d-1}^{o}) =0$, and 
the
 following long exact sequences
\begin{equation}\label{eq:MV q=2}
  0\longrightarrow H_{2,2}(L_{k,d,d-1}^{o}) \longrightarrow H_{2,2}(L_{k,d,d-1})
  \longrightarrow H_{1,1}(C_{k,d-1})\longrightarrow H_{2,1}(L_{k,d,d-1}^{o}) \longrightarrow
\end{equation}
$$\longrightarrow
H_{2,1}(L_{k,d,d-1})\longrightarrow H_{1,0}(C_{k,d-1})
\longrightarrow  H_{2,0}(L_{k,d,d-1}^{o}) \longrightarrow 0
   $$
and
\begin{equation}\label{eq:MV q=1}
  0\longrightarrow  H_{0,1}(C_{k,d-1}) \longrightarrow
  H_{1,1}(C_{k,d-1})\times H_{1,1}(L_{k,d,d-1}^{o}) 
  \longrightarrow H_{1,1}(L_{k,d,d-1})
\longrightarrow 
\end{equation}
$$   \longrightarrow H_{1,0}(C_{k,d-1})\longrightarrow H_{1,0}(C_{k,d-1})\times H_{1,0}(L_{k,d,d-1}^{o}) \longrightarrow
  0  $$
The map $H_{2,2}(L_{k,d,d-1})\to H_{1,1}(C_{k,d-1})$ is an isomorphism,
so we obtain $H_{2,2}(L_{k,d,d-1}^{o}) =0$ from
$(\ref{eq:MV q=2})$.
Next, the map $H_{2,1}(L_{k,d,d-1}) \to H_{1,0}(C_{k,d-1})$ is the zero map, since the
support of the image of any cycle is contained in disconnecting edges
of $C_{k,d-1}$. Hence we obtain
statement concerning
$H_{2,1}(L_{k,d,d-1}^{o})$ and $H_{2,0}(L_{k,d,d-1}^{o})$ 
 from
$(\ref{eq:MV q=2})$.

 The map $H_{1,0}(C_{k,d-1})\to H_{1,0}(C_{k,d-1})\times H_{1,0}(L_{k,d,d-1}^{o})$ is the identity
on the first factor, so we obtain from $(\ref{eq:MV q=1})$ the
statements about $h_{1,0}(L_{k,d,d-1}^{o})$ and $h_{1,1}(L_{k,d,d-1}^{o})$.
Since the map $H_{0,1}(C_{k,d-1}) \to H_{1,1}(C_{k,d-1})$ is the zero map, we obtain
the injectivity of the map
$H_{1,1}(\Sigma_{k,d}^{oo})\to H_{1,1}(L_{k,d,d-1})$
from $(\ref{eq:MV q=1})$.
\end{proof}

For simplicity, we denote by $(X_{k,d})_{d\ge 1}$ rather than $(X_{2,k,d})_{d\ge 1}$ 
the family of
floor composed tropical surfaces constructed in the proof of Theorem
\ref{Thm:HigherConst}
out of the family $(C^{k}_d)_{d\ge  1}$ of tropical curves contained
in the tropical plane $L_k$.
Since the  tropical surface $X_{k,1}$ is a tropical plane, it
has the following tropical Hodge diamond by Corollary~\ref{cor:hom linear}:
$$\begin{array}{ccccc}
  & & 1 & &
\\  & 0 & & 0 &
\\ 0 &
&1  & & 0
\\  & 0 & & 0 &
\\  & &1 & &
\end{array} $$

\begin{prop}\label{prop:rec hom fd}
For any integers $k\ge 1$ and $d\ge 2$, the tropical surface $X_{k,d}$ has the following
tropical Hodge diamond
$$\begin{array}{ccccc}
  & & 1 & &
\\  & 0 & &\qquad 0  \qquad&
\\ h_{2,0}(X_{k,d-1})+k\cdot g(C_{1,d-1}) &
&h_{1,1}(X_{k,d}) & & h_{0,2}(X_{k,d-1}) +k\cdot g(C_{1,d-1})
\\  & h_{2,1}(X_{k,d-1}) +k-1 & & 0 &
\\  & &1 & &
\end{array} $$
where
$$h_{1,1}(X_{k,d})=h_{1,1}(X_{k,d-1}) + k\cdot\big[d\cdot (d-1)
  +2g(C_{1,d-1})-1\big] -2(k-1)\cdot(d-2).$$
Furthermore for any $d\ge 1$, the natural map $H_{2,1}(X_{k,d}) \to H_{1,0}(C_{k,d}) $ is the zero map.
\end{prop}
\begin{proof}
  Since $H_{2,1}(X_{k,1})=0$, the map
  $H_{2,1}(X_{k,1}) \to  H_{1,0}(C_{k,1}) $ is clearly the zero map.
  We do not compute the tropical Hodge numbers with $q=0$ here, since they
  correspond to Betti numbers and have already been  computed  in Theorem~\ref{Thm:Betti}. 
We denote by $X_{k,d}^{o}$ the tropical surface $X_{k,d}$ from which we remove
the copy of the curve $C_{k,d}$ located on the boundary.
Let $d\ge 2$, and suppose that the proposition is true for $d-1$.
Since the map $H_{2,1}(X_{k,d-1}) \to H_{1,0}(C_{k,d-1}) $ is the zero map, 
by the same computation performed in the proof of Lemma
\ref{lem:remove boundary modif} we obtain that $X_{k,d-1}^o$ has the
following tropical Hodge diamond
$$\begin{array}{ccccc}
  & & 1 & &
\\  & 0 & & 0 &
\\ h_{2,0}(X_{k,d-1})+g(C_{k,d-1}) & &h_{1,1}(X_{k,d-1})-1+g(C_{k,d-1})  &\phantom{h_{2,1}(X)}  & h_{0,2}(X_{k,d-1})
\\  &  h_{2,1}(X_{k,d-1}) & & 0 &
\\  & &0 & &
\end{array} $$
We consider the same decomposition of $X_{k,d}$ as 
 in the
proof of Proposition \ref{Thm:Betti}.
By the Mayer-Vietoris theorem together with Lemmas
\ref{lem:line bd} and  \ref{lem:remove boundary modif} and Example \ref{ex:torus bd curve}, we obtain 
that $h_{1,0}(X_{k,d})=0$, and the
 following long exact sequences
\begin{equation}\label{eq:MVS q=2}
  0 \longrightarrow H_{2,2}(X_{k,d})
  \longrightarrow H_{1,1}(C_{k,d-1})\longrightarrow
  H_{2,1}(L_{k,d,d-1}^{o})\times H_{2,1}(X_{k,d-1}^{o}) \longrightarrow
  H_{2,1}(X_{k,d})\longrightarrow
\end{equation}
$$\longrightarrow H_{1,0}(C_{k,d-1}) \longrightarrow
  H_{2,0}(L_{k,d,d-1}^{o})\times H_{2,0}(X_{k,d-1}^{o})\longrightarrow H_{2,0}(X_{k,d}) \longrightarrow 0
   $$
and
\begin{equation}\label{eq:MVS q=1}
  0\longrightarrow   H_{1,2}(X_{k,d})
  \longrightarrow H_{0,1}(C_{k,d-1}) \longrightarrow
  H_{1,1}(L_{k,d,d-1}^{o})\times H_{1,1}(X_{k,d-1}^{o}) 
  \longrightarrow
\end{equation}
$$\longrightarrow   H_{1,1}(X_{k,d})
\longrightarrow H_{1,0}(C_{k,d-1})\longrightarrow 0 $$
The map $H_{2,2}(X_{k,d})\to H_{1,1}(C_{k,d})$ is clearly an
isomorphism. Furthermore the map
$H_{1,0}(C_{k,d-1}) \longrightarrow H_{2,0}(L_{k,d,d-1}^{o})\times H_{2,0}(X_{k,d-1}^{o})$ is
injective
by Lemma \ref{lem:remove boundary modif}, hence we obtain from
$(\ref{eq:MVS q=2})$ that
\begin{align*}
  h_{2,1}(X_{k,d}) &= h_{2,1}(L_{k,d,d-1}^{o})+ h_{2,1}(X_{k,d-1}^{o})
\\  &= h_{2,1}(X_{k,d-1}) +k-1
\end{align*}
and
\begin{align*}
  h_{2,0}(X_{k,d}) &=  h_{2,0}(X_{k,d-1}^{o})
\\  &= h_{2,0}(X_{k,d-1}) +g(C_{k,d-1}).
\end{align*}
The map
$H_{0,1}(C_{k,d-1}) \longrightarrow H_{1,1}(L_{k,d,d-1}^{o})\times H_{1,1}(X_{k,d-1}^{o})$ is
injective
by Lemma \ref{lem:remove boundary modif}, hence we obtain from
$(\ref{eq:MVS q=1})$ that $h_{1,2}(X_{k,d})=0$ and
\begin{align*}
  h_{1,1}(X_{k,d}) &= h_{1,1}(L_{k,d,d-1}^o)+h_{1,1}(X_{k,d-1}^o)
  \\&=  h_{1,1}(X_{k,d-1}) + k\cdot\big[d\cdot (d-1)+g(C_{1,d-1})\big]
  -(k-1)\cdot(2d-3) -1 +g(C_{k,d-1})
  \\  &= h_{1,1}(X_{k,d-1}) + k\cdot\big[d\cdot (d-1)
  +2g(C_{1,d-1})-1\big] -2(k-1)\cdot(d-2).
\end{align*}
With the exact same proof of Lemma \ref{lem:remove boundary modif}, we
obtain that  the natural map
$H_{2,1}(L_{k,d,d-1}) \to H_{1,0}(C_{k,d}) $ is the
zero map.  Hence
the map
$$H_{2,1}(X_{k,d})= H_{2,1}(L_{k,d,d-1}) \times H_{2,1}(X_{k,d-1}) \longrightarrow H_{1,0}(C_{k,d}) $$
is  the zero map, since the above Mayer-Vietoris sequence also
implies  that the map
$H_{2,1}(X_{k,d-1}) \to H_{1,0}(C_{k,d}) $ is the zero map.
\end{proof}

\subsection{Proof of Theorem \ref{thm:main surface}}
We prove the theorem by choosing $X=X_{k,d}$, and by computing its tropical
homology groups recursively on $d$ using Proposition
\ref{prop:rec hom fd}.  The theorem holds for $k=1$ by \cite{Sha13}, and so
 for all numbers $h_{p,q}(X)$ with $(p,q)\ne(1,1)$. Since we have
$$h_{1,1}(X_{k,1})=k\cdot h_{1,1}^\C(1,2) -(k-1),$$
we obtain
\begin{align*}
  h_{1,1}(X)&=k\cdot h_{1,1}^\C(d,2) - (k-1)\cdot(d-1)\cdot(d-2)-(k-1)
  \\&=h^{\C}_{1,1}(d,2) +\frac{(k-1)\cdot(d-1)\cdot(2d^2-7d+9)}{3}
  \end{align*}
as announced.
\hfill\BasicTree[1.3]{orange!97!black}{green!90!white}{green!50!white}{leaf}

\bibliographystyle{alpha}
\bibliography{biblio}
\end{document}